\documentclass[11pt]{article}
\usepackage{amsmath,amsfonts,amssymb,amsthm}
\usepackage{mathrsfs}
\usepackage{latexsym}
\usepackage[english]{babel}
\usepackage{bbm}

\usepackage{graphicx}
\usepackage[usenames,dvipsnames]{xcolor}

\usepackage{tikz}
\usepackage{enumitem}
\usepackage{multicol}

\definecolor{dkred}  {rgb}{0.58,0.05,0.3}
\definecolor{pinegreen}{rgb}{0.0, 0.47, 0.44}
\definecolor{mygreen}{rgb}{0.09, 0.45, 0.27}

\renewenvironment{proof}[1][\proofname]{{\bfseries #1.} }{\qed}

\def\<{\langle}
\def\>{\rangle}

\def\Cov{{\rm Cov\,}}

\newcommand{\field}[1]{\mathbb{#1}}

\newcommand{\DD}{\mathbb{D}}
\newcommand{\Leb}{\mathrm{Leb}}

\newcommand{\R}{\field{R}}

\newcommand{\X}{\tilde X}

\newcommand{\N}{\field{N}}

\newcommand{\C}{\field{C}}

\newcommand{\Var}{{\rm Var}}

\newcommand{\G}{{\mathcal G}}
\newcommand{\e}{{\rm e}}
\newcommand{\F}{{\mathscr{F}}}

\newcommand{\B}{{\mathscr B}}

\newcommand{\var}{\operatorname{Var}}

\def\C{{\mathcal{C}}}
\def\dTV{{d_{\mathrm{TV}}}}
\def\dW{{d_{\mathrm{W}}}}
\def\SSd{{\mathbb{S}^d}}
\def\SS2{{\mathbb{S}^2}}

\def\E{{\mathbb{ E}}}

\def\P{{\mathbb{P}}}
\def\F{{\mathscr{F}}}

\def\1{{\mathbbm{1} }}
\newtheorem{theorem}{Theorem}[section]

\newtheorem{proposition}[theorem]{Proposition}
\newtheorem{lemma}[theorem]{Lemma}

\newtheorem{definition}[theorem]{Definition}

\newtheorem{assumption}[theorem]{Assumption}

\newtheorem{remark}[theorem]{Remark}

\numberwithin{equation}{section}

\begin{document}
	
\title{Convergence in Total Variation for nonlinear functionals of random hyperspherical harmonics}
\author{
	{\sc Lucia Caramellino}\thanks{%
	Dipartimento di Matematica,
	Universit\`a di Roma Tor Vergata, and INdAM-GNAMPA - {\tt caramell@mat.uniroma2.it}}\\
	{\sc Giacomo Giorgio}\thanks{%
	Dipartimento di Matematica,
	Universit\`a di Roma Tor Vergata, and INdAM-GNAMPA - {\tt giorgio@mat.uniroma2.it}}\\
	{\sc Maurizia Rossi}\thanks{%
Dipartimento di Matematica e Applicazioni, Universit\`a  di Milano-Bicocca, and INdAM-GNAMPA  - {\tt maurizia.rossi@unimib.it}}
}
\maketitle

\begin{abstract}
Random hyperspherical harmonics are Gaussian Laplace eigenfunctions on the unit $d$-dimensional sphere ($d\ge 2$). We study the convergence in Total Variation distance for their nonlinear statistics in the high energy limit, i.e., for diverging sequences of Laplace eigenvalues. Our approach takes advantage of a recent result by Bally, Caramellino and Poly (2020): combining the Central Limit Theorem in Wasserstein distance obtained by Marinucci and Rossi (2015) for Hermite-rank $2$ functionals with new results on the asymptotic behavior of their Malliavin-Sobolev norms, we are able to establish second order Gaussian fluctuations in this stronger probability metric as soon as the functional is regular enough. Our argument requires some novel estimates on moments of products of Gegenbauer polynomials that may be of independent interest, which we prove via the link between graph theory and diagram formulas.
\end{abstract}

\smallskip

\noindent
Mathematics Subject Classification 60G60, 60B10, 42C10, 60H07.

\tableofcontents

\section{Introduction}

Random hyperspherical harmonics $\lbrace T_\ell \rbrace_{\ell\in \mathbb N}$ are Gaussian Laplace eigenfunctions on the unit $d$-dimensional sphere $\mathbb S^d$ ($d\ge 2$). They are the Fourier components of isotropic Gaussian spherical random fields,
 therefore used in a wide range of disciplines;
in particular, for $d=2$ they play a key role in cosmology -- in connection with the analysis of the Cosmic Microwave Background radiation data -- as well as in medical imaging and atmospheric sciences, see \cite[Chapter 1]{MPbook} for more details. For these reasons, in the last years the investigation of their geometry received a great attention, in particular the asymptotic behavior, for large eigenvalues (as $\ell\to +\infty$), of their nonlinear statistics $\lbrace \X_\ell\rbrace_{\ell\in \mathbb N}$, see \cite{MW11, MW14, ROS20, Dur16, CM18, Tod19, MRW20} and the references therein. 
The main goal of most of these papers is to study first and second order fluctuations for $\X_\ell$ to be some geometric functional of the excursion sets of $T_\ell$, such as the so-called Lipschitz-Killing curvatures
\cite[Section 6.3]{AT} that in dimension $2$ are the area, the boundary length and the Euler-Poincar\'e characteristic. Hence it is clear that $\X_\ell$ may be a function of the sole $T_\ell$ (in the case of the excursion measure for instance) or a function of $T_\ell$ and its derivatives. 

The above mentioned references take advantage of Wiener-It\^o theory, the random variables $\lbrace \X_\ell \rbrace_{\ell\in \mathbb N}$ being square integrable functionals of Gaussian fields. In this framework, the techniques developed allow one to establish Central Limit Theorems (CLTs) via a powerful combination of chaotic expansions and fourth moment theory by Nourdin and Peccati \cite{NP12}. It is well known that the link between Malliavin calculus and Stein's method established by these two authors permits to get estimates on the rate of convergence to the limiting Gaussian law in various probability metrics \cite[Appendix C.2]{NP12}, at least when a finite number of chaoses are involved. For general functionals instead, the so-called second order Poincar\'e inequality \cite{NPR09} may be evoked, even in its improved version \cite{Vid19}.   

However, the existing results in the literature for the above mentioned geometric functionals $\lbrace \X_\ell\rbrace_{\ell\in \mathbb N}$ (which do have an infinite chaos expansion) of random hyperspherical harmonics $\lbrace T_\ell \rbrace_{\ell\in \mathbb N}$ only deal with the Wasserstein distance, see e.g. \cite{ROS20, CM18, Ros19}. The typical situation is a single chaotic component dominating the whole series expansion, entailing the Wasserstein distance to be controlled by the square root of the fourth cumulant of this leading term plus the $L^2(\mathbb P)$-norm of the series tail. Moreover, generally there are no information on the optimal speed of convergence.

A natural question is whether or not these results could be upgraded to stronger probability metrics. Here we address this issue, indeed we are interested in quantitative CLTs in Total Variation distance \cite[Section C.2]{NP12} for nonlinear statistics $\lbrace \X_\ell\rbrace_{\ell\in \mathbb N}$ of random hyperspherical harmonics $\lbrace T_\ell\rbrace_{\ell\in \mathbb N}$ in the high energy limit (as $\ell\to+\infty$). We are able to solve the problem for integral functionals of the sole $T_\ell$, that are regular enough in the Malliavin sense, by taking advantage of a recent result in \cite{BCP19}. 
In this paper, the authors prove some regularization lemmas that enable one to upgrade the distance of convergence from smooth Wasserstein to Total Variation (in a quantitative way) for any sequence of random variables which are smooth and non-degenerate in some sense. The price to pay is to control the smooth Wasserstein distance between the sequence of their Malliavin covariance matrices and its limit, that however does \emph{not} need to be the Malliavin covariance  matrix of the limit. Remarkably, this technique requires neither the sequence of random variables of interest to be functionals of a Gaussian field nor the limit law to be Normal, situations that naturally occur since the underlying randomness may be not Gaussian \cite{BCPzeri, CNN20} or related functionals may show non-Normal second order fluctuations \cite{MPRW16}.

%

Let us write down explicitly our functional of interest: we consider
	\begin{equation*}
	\X_{\ell}=\frac{X_\ell-\E[X_\ell]}{\sqrt{\var(X_\ell)}} \quad\mbox{where}\quad X_\ell := \int_{\mathbb S^d} \varphi(T_\ell(x)) dx,
	\end{equation*}
 	$\varphi:\mathbb R\to \mathbb R$ being square integrable w.r.t. the Gaussian density.  In \cite{ROS20}, the authors prove that, under mild assumptions, the above functional $\X_{\ell}$ converges in Wasserstein distance towards a Gaussian random variable as $\ell\to +\infty$ ;
 in order to strengthen this result, in light of \cite{BCP19}, we need to investigate the asymptotic behavior of the Malliavin covariance of $\X_\ell$, that we denote by $\sigma_\ell$. Under some additional regularity properties on the function $\varphi$ which are needed to ensure the existence of Malliavin derivatives of $\X_\ell$, we are able to prove the convergence in Wasserstein distance of $\sigma_\ell$ towards a non-degenerate deterministic limit, that together with the uniform boundedness of Malliavin-Sobolev norms of $\X_\ell$ guarantees the convergence in Total Variation distance for $\X_\ell$. 
To the best of our knowledge, ours is the first quantitative Limit Theorem in Total Variation distance for nonlinear functionals of random hyperspherical harmonics having an infinite chaotic expansion, generalizing in particular the work 
\cite{ROS20}. 

As a bonus, we gain some new results on the asymptotic behavior of Malliavin derivatives of these functionals, and some novel estimates on the moments of products of powers of Gegenbauer polynomials (the latter describing the covariance structure of the random hyperspherical harmonics $\lbrace T_\ell\rbrace_{\ell\in \mathbb N}$) thus extending some formulas in \cite{Mar08, Ros19}. 
For our investigation we exploit an explicit link between the diagram formula for moments of Hermite polynomials and the graph theory, inspired by \cite{Mar08}. In particular, we extrapolate a graph from each of these diagrams and use the fact that every connected graph can be covered by a tree, eventually studying only the contribution coming from these trees. In order to make the reading pleasant and smooth, we collect the proofs of these key results on Gegenbauer integrals in Appendix \ref{Diag} and Appendix \ref{sect:graph}.

Finally, it is worth stressing that in the context of Gaussian approximations for random variables that are functionals of an underlying Gaussian field, the second order Poincar\'e inequality by Vidotto \cite{Vid19}  has led to quantitative CLTs for nonlinear functionals of stationary Gaussian fields related to
the Breuer-Major theorem, with
presumably optimal rates of convergence in Total Variation distance. However, we choose to exploit the technique developed in \cite{BCP19} with a view to a subsequent generalization of our result to the interesting case of random eigenfunctions of the standard flat torus (arithmetic random waves), where the attainable limit laws include linear combinations of independent chi-square distributed random variables \cite{MPRW16, Cam19}. Moreover, it turns out that in order to obtain fruitful bounds via the second order Poincar\'e inequality for the Gaussian approximation of our functional of interest $\tilde X_\ell$, the estimates on moments of products of powers of Gegenbauer polynomials should be much finer than those required by the approach developed in \cite{BCP19} (the one that we follow).

\subsection{Notation} 

Throughout this manuscript we denote with $\nu$ the standard Gaussian law on $\R$ and with $Z\sim \mathcal N(0,1)$ a standard Gaussian random variable (r.v.).  When we will speak about Malliavin calculus and chaos expansion based on $Z$, we just intend the classical one dimensional approach in the space $L^2(\nu):=L^2(\R,\B(\R), \nu)$ (see e.g. \cite[Chapter 1]{NP12}), where $\B(\R)$ denotes the Borel-$\sigma$ field on the real line. In particular, we will denote by $L$ and $D^ k$ (for integers $k\ge 1$) the Ornstein-Uhlenbeck operator and  the $k$th order Malliavin derivative, respectively. As usual, $Dom(L)$ and $\DD^ {k,p}$ (for $p\geq 1$) will stand, respectively, for the set of random variables measurable w.r.t. $\sigma(Z)$ on which $L$ is well defined and that are derivable in Malliavin sense up to order $k$, whose derivatives all belong to $L^p(\mathbb P):=L^p(\Omega, \mathcal F, \mathbb P)$. Here and in what follows $(\Omega, \mathcal F, \mathbb P)$ will denote a probability space and without loss of generality we may assume the random objects in this paper are defined on this common probability space.

\emph{Conventions.} In this paper we set $\mathbb N:=\lbrace 0,1, 2\dots \rbrace$ and $\N^*=\N\setminus \{0\}$. Given two sequences of positive numbers $\{a_n\}_{n\in \mathbb N}$ and $\{b_n\}_{n\in \mathbb N}$, we write $a_n\sim b_n$ if $\lim_{n\to\infty}\frac {a_n}{b_n}=1$,
$a_n= O(b_n)$ if $\{\frac {a_n}{b_n}\}_n$ is asymptotically bounded and $a_n= o(b_n)$ if $\lim_{n\to\infty}\frac {a_n}{b_n}=0$. 

\subsection*{Acknowledgements} The authors would like to thank Domenico Marinucci for useful discussions. 
L.C. and G.G. acknowledge support from the MIUR \emph{Excellence} Project MatMod@TOV awarded to the Department of Mathematics, University of Rome ``Tor Vergata'', CUP E83C23000330006, and the Project \textit{Asymptotic Properties in Probability}, CUP E83C22001780005. 
The research of M.R. has been supported by the ANR-17-CE40-0008 Project \emph{Unirandom}.
L.C. and M.R. acknowledge support from the Beyond Borders Project  AMP CUP E89C20000680005. 
\section{Motivations and main results}

Let $d\geq 2$ be a positive integer, we denote by $\mathbb S^d\subset \R^{d+1}$ the $d$-dimensional unit sphere. Accordingly, we set $\B(\mathbb S^d)$ to be its Borel $\sigma$-field and we write $\Leb(dx)=dx$ for the Lebesgue measure on $(\mathbb S^d, \B(\mathbb S^d))$. It is known that $\int_{\mathbb S^d} dx =\mu_d$ where 
$$
\mu_d=\frac{2\pi^{\frac{d+1}{2}}}{\Gamma(\frac{d+1}{2})},
$$
$\Gamma$ being the Gamma function. Let $\Delta_{\mathbb S^d}$ denote the spherical Laplacian operator on $\mathbb S^d$, we are interested in non-trivial solutions of the Helmholtz equation
\begin{equation}\label{helmholtz}
\Delta_{\mathbb S^d} f =- \lambda f,
\end{equation}
where $f:\mathbb S^d \to \mathbb R$ and $\lambda\ge 0$. It is known that the Laplacian eigenvalues are of the form $-\lambda = -\lambda_{\ell;d} = -\ell(\ell + d-1)$ for $\ell\in \mathbb N$, and the dimension of the $\ell$-th eigenspace is, for $\ell \in \mathbb N^*$,
$$
n_{\ell;d} = \frac{2\ell+d-1}{\ell} \binom{\ell+d-2}{\ell-1}
$$ 
($n_{0;d} = 1$). Notice that
$
n_{\ell;2}=2\ell + 1$ and 
\begin{equation}\label{n-ell}
n_{\ell;d} \sim \frac{2}{(d-1)!} \ell^{d-1}\quad \text{as}\quad \ell\to +\infty.
\end{equation}
We choose the family of real-valued hyperspherical harmonics \cite[Section 9.3]{VK93}
$(Y_{\ell,m;d})_{m=1}^{n_{\ell;d}}$ as orthonormal system of the $\ell$-th eigenspace. We recall that they are the restriction to the sphere of harmonic polynomials of degree $\ell\in\N$ in $d+1$ variables. Note that every eigenfunction $f=f_{\ell}$ in (\ref{helmholtz}) of eigenvalues $-\lambda_{\ell;d}$ can be written in the form
\begin{equation}\label{eigenfunction}
f_\ell(x)=\sqrt{\frac{\mu_d}{n_{\ell;d}}}\sum_{m=1}^{n_{\ell;d}} c_{\ell,m} Y_{\ell,m;d}(x),\quad x\in \mathbb{S}^d,
\end{equation}
where $c_{\ell,m}$ are real numbers, and $\sqrt{\mu_d}/\sqrt{n_{\ell;d}}$ is a normalizing factor whose role will be clarified just below. In this paper we are interested in \emph{random} eigenfunctions, i.e. linear combinations of hyperspherical harmonics as in (\ref{eigenfunction}) with $c_{\ell,m}$ replaced by random coefficients.  

\subsection{Random hyperspherical harmonics}

 For $\ell \in \mathbb N^*$, we define the $\ell$th random hyperspherical harmonic $T_\ell$ on $\mathbb S^d$ through
\begin{equation}\label{Tl}
T_\ell(x)=\sqrt{\frac{\mu_d}{n_{\ell;d}}}\sum_{m=1}^{n_{\ell;d}} a_{\ell,m} Y_{\ell,m;d}(x),\quad x\in \mathbb{S}^d,
\end{equation} 
where $(a_{\ell,m})_{m=1}^{n_{\ell;d}}$ are standard Gaussian i.i.d. random variables in $\R$. Then,
$$
T_\ell : \Omega \times \mathbb S^d \longrightarrow \R; \quad (\omega, x)\mapsto T_\ell(\omega, x)
$$
 is an isotropic and centered Gaussian random field on $\mathbb S^d$ (we will omit the dependence of $\omega$ in $T_\ell(\omega,x)$, as usual). The isotropy property for $T_\ell$  means that the random fields $T_\ell(\cdot)$ and $T_\ell(g \cdot)$ are equal in law for every $g\in SO(d+1)$, the special orthogonal group of $(d+1)\times (d+1)$-matrices. Indeed, the covariance kernel of $T_\ell$ is (by the addition theorem for hyperspherical harmonics  \cite[Section 9.6]{AAR99}) 
 \begin{equation}\label{covT}
 \Cov(T_\ell(x),T_\ell(y)) = \frac{\mu_d}{n_{\ell;d}}\sum_{m=1}^{n_{\ell;d}} Y_{\ell,m;d}(x) Y_{\ell,m;d}(y)= G_{\ell;d}(\<x,y\>), \quad x,y\in \mathbb S^d,
 \end{equation}
 where $G_{\ell;d}$ denotes the $\ell$th Gegenbauer polynomial \cite[\S 4.7]{Sze39} (for $d=2$, $G_{\ell; 2}\equiv P_\ell$, that is, the Legendre polynomial of degree $\ell$), and $\<x,y\>=\cos d(x,y)$ where $d(x,y)$ is the geodesic distance between points $x,y\in \mathbb S^d$. Gegenbauer polynomials are orthogonal polynomials on the interval $[-1,1]$ w.r.t. the weight $(1-t^2)^{(d-1)/2}$, and we normalize them such that 
 $$
 G_{\ell;d}(1) = 1. 
 $$
 In particular, this implies that $\Var(T_\ell(x)) = 1$ for every $x\in \mathbb S^d$. Note also that, by construction (\ref{Tl}), $T_\ell$ a.s. satisfies the Helmholtz equation (\ref{helmholtz}) with eigenvalue $-\lambda_\ell = -\ell(\ell+d-1)$. 

\subsection{Statistics of random hyperspherical harmonics}

We are interested in functionals of random hyperspherical harmonics of the type 
\begin{equation}\label{Xl}
X_\ell := \int_{\mathbb S^d} \varphi(T_\ell(x)) dx,
\end{equation}
where $\varphi\in L^2(\nu)$. In particular, we study the asymptotic behavior of the sequence of random variables $\lbrace X_\ell\rbrace_{\ell\in \mathbb N}$ as $\ell\to +\infty$ by means of chaotic decompositions \cite[\S 2.2]{NP12}: if $Z\sim \mathcal N(0,1)$, then $\varphi(Z)$ can be written as an orthogonal series in $L^2(\mathbb P)$ as follows 
\begin{equation}\label{chphi}
\varphi(Z)=\sum_{q\geq 0} \frac{b_q}{q!} H_q(Z),\quad \text{ where }\quad b_q:=\E[\varphi(Z)H_q(Z)],
\end{equation}
and from now on $H_q$ denotes the Hermite polynomial of order $q$ --  properties of Hermite polynomials and the Wiener chaos expansion will be recalled in Section \ref{background}.  Substituting (\ref{chphi}) into (\ref{Xl}) gives the chaotic expansion for $X_\ell$:
 \begin{equation}\label{chXl}
X_\ell = \sum_{q\ge 0} X_\ell[q],\quad \text{ where }\quad  X_\ell[q]:= \frac{b_q}{q!} \int_{\mathbb S^d} H_q(T_\ell(x))\,dx,
\end{equation}
see Section \ref{Conv0} for more details. Note that the term corresponding to $q=1$ in the series (\ref{chXl}) is null, indeed $H_1(x)=x$ and by orthogonality properties of hyperspherical harmonics, we have $\int_{\mathbb S^d} Y_{\ell,m;d}(x)\,dx = 0$ for $\ell\in \mathbb N^*$.

By standard properties of Hermite polynomials \cite[Section 2.2]{NP12} it is immediate to check that 
\begin{equation}\label{first_moments}
\mathbb E[X_\ell] = \mathbb E [X_\ell[0]]= \mathbb E[\varphi(Z)] \mu_d,\quad \Var(X_\ell) = \sum_{q\ge 2}\frac{b_q^2}{q!} \int_{(\mathbb S^d)^2} G_{\ell;d}(\<x,y\>)^q\,dx dy. 
\end{equation}
\begin{remark}\rm
By the symmetry property of Gegenbauer polynomials, i.e., 
$$
G_{\ell;d}(t)=(-1)^\ell G_{\ell; d}(-t),
$$
 if both $\ell$ and $q$ are odd the $q$th moment of $G_{\ell;d}$ vanishes. From now on, we take only even $\ell$, hence by $\ell\to +\infty$ we mean \emph{as $\ell$ goes to infinity along even $\ell$}. 
\end{remark}
  It is known (see Proposition \ref{GegProp} for more details) that the $q$th moment of Gegenbauer polynomials $G_{\ell;d}$ behaves, as $\ell\to +\infty$, as $1/n_{\ell;d}$ up to positive constants for $q=2$ while for $q\ge 3$ it is $o(1/n_{\ell;d})$. 
  
Now we recall the notion of \emph{Hermite rank}: we say that $X_\ell$ has Hermite rank $k\ge 2$ if $k$ is the smallest $q\ge 2$ for which $X_\ell[q]\ne 0$ (cf. (\ref{chXl})). Then we recall the following result for Hermite-$2$ rank statistics. Before, we need to introduce some more notation: we define the standardized statistic 
  \begin{equation}\label{tildeXl}
\X_\ell=\frac{ X_\ell-\E[X_\ell]}{\sqrt{\var(X_\ell)}},
\end{equation} 
and the Wasserstein distance between two integrable random variables $X,Y$:
\begin{equation}
\dW(X,Y) := \sup_{h\in \text{Lip}(1)} \left |\mathbb E[h(X)] - \mathbb E[h(Y)] \right |,
\end{equation}
where $\text{Lip}(1)$ denotes the space of functions $h:\mathbb R\to \mathbb R$ which are Lipschitz continuous with Lipschitz constant $\le 1$. 
 \begin{theorem} [Theorem 1.7 in \cite{ROS20}]\label{ROS}
Let $\varphi$ be as in \eqref{chphi} such that $b_2\neq 0$. Then, as $\ell\to +\infty$,
\begin{equation}\label{varXl}
\Var(X_\ell) \sim \frac{b_2^2}{2} \frac{(\mu_d)^2}{n_{\ell;d}}, 
\end{equation}
and moreover 
\begin{equation}\label{dW-th}
\dW(\X_{\ell}, Z)=O\big(\ell^{-\frac{1}{2}}\big).
\end{equation}
\end{theorem}
Theorem \ref{ROS} only deals with Wasserstein distance, and gives no information on the speed of convergence in stronger probability metrics such as Total Variation. We recall that, for random variables $X$ and $Y$
\begin{equation}\label{TVdistance}
\dTV(X,Y) := \sup_{A\in \B(\mathbb R)}\left | \mathbb P(X\in A) - \mathbb P(Y\in A)\right |.
\end{equation}

The main goal of this paper is to strenghten and upgrade Theorem \ref{ROS} from Wasserstein to Total Variation distance for suitably regular nonlinear functionals of $T_\ell$ of the form (\ref{Xl}). 

\subsection{Statement of the main results} 

The assumptions on $\varphi$ in Theorem \ref{ROS} are rather weak: it suffices that $\varphi$ is a square integrable function w.r.t. the Gaussian measure $\nu$ and $b_2\neq 0$.  In order to investigate the convergence for $X_\ell$ towards the Gaussian law in Total Variation distance, we need $\varphi(Z)$ to satisfy some additional regularity properties in the Malliavin sense.  These are summarized in the following condition. 


\begin{assumption}\label{ASSUMPTION}
	Let $\varphi(Z)$ fulfill  \eqref{chphi}. We assume that $b_2\neq 0$. Moreover,  
	$\varphi(Z)\in Dom(L)$ and 	
	$\varphi(Z), L\varphi(Z)\in \cap_{k\geq 0}\cap_ {p\geq 2}\DD^{k,p}$, that is,  for every $k\in \N$ and $p\geq 2$ the $k$th order Malliavin derivative of $\varphi(Z)$ and of $L\varphi(Z)$, given by
	\begin{equation}\label{Dkphi}
	D^ k\varphi(Z)=\sum_{q\geq k} \frac{b_{q}}{(q-k)!} H_{q-k}(Z)
	\quad \text{ and } \quad D^ kL\varphi(Z)=-\sum_{q\geq k} q\,\frac{b_{q}}{(q-k)!} H_{q-k}(Z),
	\end{equation}
 do exist and belong to $L^p(\mathbb P)$.
	Furthermore, the same properties are satisfied by the function $\phi\in L^2(\nu)$ defined by
	\begin{equation}\label{Tpsi}
	\phi(z)= \sum_{q\geq 2} \frac{|b_q|}{q!} H_{q}(z),
	\end{equation}
	that is, $\phi(Z)\in Dom(L)$ and 	 $\phi(Z), L\phi(Z)\in  \cap_{k\geq 0}\cap_ {p\geq 2}\DD^{k,p}$: for $k\in\N$ and $p\geq 2$, 
	\begin{equation}
	\label{DkTphi}
	D^ k\phi(Z)=\sum_{q\geq 2\vee k} \frac{|b_{q}|}{(q-k)!} H_{q-k}(Z)\quad
	\text{ and }\quad
	D^ kL\phi(Z)=-\sum_{q\geq 2\vee k} q\frac{|b_{q}|}{(q-k)!} H_{q-k}(Z)
	\end{equation}
	both belonging to $L^p(\mathbb P)$. 
	%
	%
\end{assumption}
From now on we assume that Assumption \ref{ASSUMPTION} holds. The requested Malliavin regularity will not be really surprising once the mathematical tools we are going to use will become clear (namely, the use of Proposition \ref{Main}).
Let us  give right away a condition allowing $\varphi$ to satisfy Assumption \ref{ASSUMPTION}.
	\begin{proposition}\label{exponential}
	Suppose  that there exist $C, R>0$ such that $|b_q|\leq CR^ q$ for every $q\geq 0$ in (\ref{chphi}). Then Assumption \ref{ASSUMPTION} holds.
\end{proposition}
The proof of Proposition \ref{exponential} is postponed to Appendix \ref{appendix1}.

As a meaningful example, let $t\in\R$ denote any parameter and set $\varphi(z)=\e^{tz}$. Then $\varphi$ satisfies Assumption \ref{ASSUMPTION}, as a consequence of Proposition \ref{exponential} and of the well known representation
	$$
	\e^{tz}=\e^ {\frac{t^2}2}\sum_{q\geq 0} \frac{t^ q}{q!}\, H_q(z), \quad z\in\R.
	$$
	Notice that the above function provides a nonlinear functional that does not have a finite chaos expansion.

We are now in a position to state the main result of this paper.

\begin{theorem}\label{mainThm}
Let $\varphi$ satisfy Assumption \ref{ASSUMPTION}, then, for any $0< \varepsilon < 1$, as $\ell\to +\infty$,
\begin{equation}\label{TVrate}
\dTV(\X_{\ell}, Z) =O_\varepsilon\big ( \ell^{-\frac{1-\varepsilon}{2}} \big)
\end{equation}
where $O_\varepsilon$ means that the constants involved in the $O$-notation depend on $\varepsilon$. 
\end{theorem}
	It is worth noticing that, for any dimension $d$, the upper bound in (\ref{TVrate}) for the Total Variation distance coincides (morally) with the quantitative bound of $\dW(\X_{\ell}, Z)$ in \eqref{dW-th}. So, conditionally on Theorem \ref{ROS} (and the use of Proposition \ref{Main} below), our result cannot be improved.

To the best of our knowledge, Theorem \ref{mainThm} is the first result on the convergence of statistics of random hyperspherical harmonics (in particular having an infinite chaos expansion) in Total Variation distance. For $\varphi = H_q$ the $q$th Hermite polynomial with $q\ge 2$, or for $\varphi$ equal to a linear combination of  such Hermite polynomials, bounding from above $\dTV(\X_{\ell}, Z)$ is an application of the fourth moment theorem by Nourdin and Peccati, see \cite{MW14} for results in the two-dimensional case and \cite{ROS20, Ros19} for higher dimensions. 

An intermediate key step to prove our main result relies on the investigation of the asymptotic behavior of the sequence of Malliavin derivatives of $X_\ell$; we stress that this analysis is new and leads to some results of independent interest, see Proposition \ref{Conv0} and Proposition \ref{UnifLim} for more details. 

Besides the case of higher Hermite rank functionals, that we do believe it can be dealt with by using the same approach as the one developed for the proof of Theorem \ref{mainThm} though involving heavier computations, we leave as a topic for future research the interesting case of the indicator function: for $u\in \mathbb R$, 
\begin{equation*}
\varphi(z) = \1_{[u,+\infty)}(z),\quad z\in \mathbb R,
\end{equation*}
thus $X_\ell$ is the so-called excursion area at level $u$, see \cite{MW14}. Indeed, $\varphi(Z)$ is \emph{not} derivable in Malliavin sense, and the Assumption \ref{ASSUMPTION} is not satisfied. 

\section{Proofs of the main results}

In this Section we explain the main ideas behind our argument, eventually giving the proof of our main result.

\subsection{On the proofs}\label{sect:proofs}

To show the main ideas of the proof of Theorem \ref{mainThm} and of the results that we are going to use, we need to introduce some properties associated with the Malliavin regularity of the random variables at hands. We give here a result developed in \cite{BCP19}, holding in a \textit{purely abstract} Malliavin calculus setting (see \cite[Section 2.1]{BCP19}), that is, based on a random noise that does not need to be Gaussian (see e.g. the one used in \cite{BCPzeri}). Let us resume it here briefly.	First of all, it is assumed that the following ingredients are given:
	\begin{itemize}
		\item a set $\mathcal{E}\subset \cap_{p\geq 2}L^ p(\Omega)$ such that for every $n\in\N^*$, $f\in C_p^ \infty(\R^ n)$ and $F=(F_1,\ldots, F_n)\in\mathcal{E}^ n$ then $f(F)\in\mathcal{E}$ (so, $\mathcal{E}$ is an algebra);
		\item a Hilbert space $\mathcal{H}$, whose inner product and associated norm will be denoted by $\<\cdot,\cdot\>_{\mathcal{H}}$ and $|\cdot|_{\mathcal{H}}$ respectively; we let $L^ p(\Omega;\mathcal{H}) $ stand for the set of the r.v.'s taking values in $\mathcal{H}$ whose norm has moment of order $p$. 
	\end{itemize}
	%
	%
	In this environment, it is assumed that there exist two linear operators 
	$$
	D\,:\,\mathcal{E}\to \cap_{p\geq 2}L^ p(\Omega;\mathcal{H}) 
	\quad\mbox{and}\quad
	L\,:\,\mathcal{E}\to \mathcal{E}
	$$
	such that
	\begin{itemize}
		\item[\textbf{(M1)}] for every $F\in\mathcal{E}$ and $h\in\mathcal{H}$, $D_hF:=\<DF,h\>_{\mathcal{H}}\in \mathcal{E}$;
		\item[\textbf{(M2)}] 
		for every $n\in\N^*$, $f\in C_p^ \infty(\R^ n)$ and $F=(F_1,\ldots, F_n)\in\mathcal{E}^ n$ one has 
		$$
		Df(F)=\sum_{i=1}^n \partial_{x_i}f(F)DF_i\in\mathcal{E};
		$$
		\item [\textbf{(M3)}] for every $F,G\in \mathcal{E}$ one has 
		$\E[LF\, G]=-\E[\<DF,DG\>_\mathcal{H}]=\E[F\, LG]$.
	\end{itemize}
	Thus, we recognize that these are settings and properties typically fulfilled in Malliavin calculus (but not in \textit{any} Malliavin calculus framework - for example this is not in the case of jump processes, where the chain rule \textbf{(M2)} does not hold in general, see e.g. the discussion and the references quoted  in \cite[Section1]{BCP19}). Hence, we call $D$ the Malliavin derivative and $L$ the Ornstein-Uhlenbeck operator. The higher order Malliavin derivatives can be defined straightforwardly: for $k\geq 2$, 
	$$
	D^ k\,:\, \mathcal{E}\to \cap_ {p\geq 2}L^p(\Omega;\mathcal{H}^ {\otimes k})
	$$ 
	is the multilinear operator such that for every $h_1,\ldots,h_k\in \mathcal{H}$ and $F\in\mathcal{E}$,
	$$
	D^ k_{h_1,\ldots,h_k}F:= \<D^ kF, h_1\otimes\cdots\otimes h_k\>_{\mathcal{H}^ {\otimes k}}=
	D_{h_k}D^ {k-1}_{h_1,\ldots,h_{k-1}}F.
	$$ 
	Notice that, when dealing with a concrete Malliavin calculus, one can choose $\mathcal{E}$ either the set of the simple functionals or  the set $\DD^ \infty$ of the r.v.'s  whose Malliavin derivative of any order does exist and has finite moment of any power.

In order to introduce the result in \cite{BCP19} that we are going to use, we first need to define the involved Malliavin-Sobolev norms: for $F=(F_1,\ldots,F_n)\in \mathcal{E}^ n$, we set
	\begin{equation}\label{mall0}
	|F|_{1,q}=\sum_{k=1}^q\sum_{i=1}^ n|D^ kF_i|_{\mathcal{H}^{\otimes k}},\quad |F|_q=|F|+|F|_{1,q},\quad \|F\|_{k,p}=\| |F|_k\|_p,
\end{equation}
where $\|\cdot \|_p$ is the standard norm in $L^p(\Omega)$. Then, for $k\in\N^ *$ and $p\geq 2$, we set
$$
\DD^ {k,p}=\overline{\mathcal{E}}^ {\|\cdot\|_{k,p}}\quad \mbox{and}\quad \DD^ {k,\infty}=\cap_ {p\geq 2}\DD^{k,p}.
$$
We also extend the operator $L$ in the usual way: for $F=(F_1,\ldots,F_n)\in \mathcal{E}^ n$, we set $LF=(LF_1,\ldots,LF_n)$ and $\|F\|_{\mathrm{OU}}=\|F\|_2+\|LF\|_2$. And we define $Dom(L)=\overline{\mathcal{E}}^ {\|\cdot\|_{{\mathrm{OU}}}}$.

Now,  fix $q\in\N$ and $F=(F_1,\ldots,F_n)\in (\DD^ {q+1,\infty})^n$. If $F=(F_1,\ldots,F_n)\in (Dom(L))^n$ and $LF=(LF_1,\ldots,LF_n)\in (\DD^ {q,\infty})^n$, the following quantities are well posed:
	\begin{equation}\label{mall}
	\begin{array}{ll}
	&\mathcal{C}_q(F)=\big(|F|_{1,q+1}+|LF|_{q}\big)^ q\big(1+|F|_{1,q+1}\big)^ {4nq},\\
    &\mathcal{C}_{q,p}(F)=\|\mathcal{C}_q(F)\|_ p,\\
    &\mathcal{Q}_{q}(F)=\mathcal{C}_{q,2}(F)\|(\det \sigma_F)^ {-1}\|_{2q}^ q,
	%
	\end{array}
	\end{equation}
	in which $p\geq 2$ and $\sigma_F$ is the Malliavin covariance matrix of $F$, that is,
	\begin{equation}\label{MallCov}
	(\sigma_F)_{i,j}=\<DF_i,DF_j\>_{\mathcal{H}},\quad i,j=1,\ldots, n.
	\end{equation}
	Notice that the quantity   $\mathcal{C}_{q,p}(F)$, respectively $\mathcal{Q}_{q}(F)$, in \eqref{mall} is in principle well posed whenever $F_i\in\DD^ {q+1,\bar p}$ for a suitable $\bar p\geq p$, respectively $\bar p\geq 2$.

We are now ready to state the result in \cite{BCP19} on which our asymptotic analysis will be based:
		
\begin{proposition}\label{Main}
Let $F$ and $G$ be random vectors in $\R^ n$ such that
$$
M_q(F,G):=\mathcal{C}_{q,1}(F)+\mathcal{Q}_{q}(G)<\infty, 
$$
for every $q\geq 1$. Let $U>0$ be a real random variable such that $\|U^{-1}\|_ q<\infty$ for every $q\geq 1$.
Then for every $\varepsilon>0$ there exist $C_\varepsilon>0$ and $q_\varepsilon>1$ such that
$$
\dTV(F,G)\leq C_\varepsilon\big(M_{q_\varepsilon}(F,G)+\|U^ {-1}\|_{2/\varepsilon}\big)\big(\dW(F, G)+\dW(\det \sigma_{F}, U)\big)^{1-\varepsilon}.
$$

\end{proposition}

This is actually \cite[Proposition 3.12]{BCP19}, see in particular (3.30), with the choice $p=p'=1$ (remark that, as it immediately and clearly follows from the proof, there is a misprint in the requests therein: it is erroneously asked that $\mathcal{C}_{q,1}(G), \mathcal{Q}_{q}(F)<\infty$ instead of $\mathcal{C}_{q,1}(F), \mathcal{Q}_{q}(G)<\infty$).

Our plan is to use Proposition \ref{Main} with $F=\X_{\ell}$ and $G=Z\sim \mathcal N(0,1)$.  Indeed, in our framework, the underlying Hilbert space is $H=L^2(\mathbb S^d, \B(\mathbb S^d), \Leb)$, the random eigenfunction $T_\ell$ admitting the isonormal representation (\ref{spR}).  Thus
\begin{equation}\label{sigmal}
\sigma_\ell = \sigma_{\tilde X_\ell} =\int_{\mathbb S^d} |D_y \tilde X_\ell |^2\,dy.
\end{equation}
First, Assumption \ref{ASSUMPTION} will guarantee that all the involved Malliavin functionals are well defined (we will give more details about Malliavin calculus for Gaussian random fields in Section \ref{sect:malliavin}). 
Theorem \ref{ROS} already ensures that $\dW(\X_{\ell}, Z)\to 0$ (giving also an estimation of the speed of convergence). Therefore, we obtain the stronger convergence $\dTV(\X_{\ell}, Z)\to 0$ (together with a useful upper bound on the rate), once we prove that:
\begin{enumerate}
	\item[\textbf{(H1)}] there exists a \textit{deterministic} $U>0$ such that $\dW(\sigma_{\ell}, U)\to 0$ with some speed, 
	\item[\textbf{(H2)}] for every $q\geq 1$, $\sup_{\ell }M_q(\X_{\ell},Z)<\infty$, where $M_q(\tilde X_\ell, Z)$ is defined in Proposition \ref{Main}.
\end{enumerate}

\subsection{Proof of Theorem \ref{mainThm}}

Concerning \textbf{(H1)}, we will prove the following key result. 

\begin{theorem}\label{Conv0}
Let $\sigma_\ell$ be the Malliavin covariance of $\X_\ell$. Under Assumption \ref{ASSUMPTION}, we have 
$$
|\E[\sigma_\ell ]- 2 | = O\left (\eta_{\ell;d}\right )
\quad \mbox{and}\quad
\Var(\sigma_\ell )= O\left (\ell^{-1}\1_{d=2}+ \ell^{-(d-1)/2}\1_{d\geq 3}\right ),
 $$
 where 
 \begin{equation*}
 \eta_{\ell;d} =  \1_{d=2}\left ( \1_{b_4\ne 0}\frac{\log \ell}{\ell} +\1_{b_4=0}\frac{1}{ \ell}\right )  + \frac{1}{\ell} \1_{d\ge 3}.
 \end{equation*}
\end{theorem}

As for \textbf{(H2)}, it is enough to prove that, uniformly in $\ell$, all the moments of the main Malliavin operators involved in $M_q(\X_{\ell},Z)$ are bounded.  This is why we will prove the following result.
\begin{proposition}\label{UnifLim}
Under Assumption \ref{ASSUMPTION}, for every $k \in \N$ and $n \geq 1$, there exists $\tilde C_{n,k,d}>0$ such that
$$
\sup_{\ell \mathrm{\ even}}\E[|D^{(k)} \X_\ell|^n_{\mathcal{H}^{\otimes k}}]\leq \tilde C_{n,k, d}
\quad\mbox{and}\quad
\sup_{\ell \mathrm{\ even} }\E[|D^{(k)} L\X_\ell|^n_{\mathcal{H}^{\otimes k}}]\leq \tilde C_{n,k, d}.
$$ 
\end{proposition}

We postpone the proofs of Proposition \ref{Conv0} and of Proposition \ref{UnifLim} to Sections constructed ad hoc (see Section \ref{sect:Conv0} and Section \ref{sect:UnifLim} respectively). Based on such results, the proof of the CLT in Total Variation distance (Theorem \ref{mainThm}) follows. 

\medskip

\begin{proof}[Proof of Theorem \ref{mainThm} assuming Propositions \ref{Conv0} and \ref{UnifLim}]
We use Proposition \ref{Main} with $F=\X_{\ell}$, $G=Z$ and $U=2$. We have
$$
\dW(\sigma_\ell, 2)\leq \|\sigma_\ell- 2\|_1\leq \|\sigma_\ell- 2\|_2\leq \Var(\sigma_\ell)^{1/2}+|\E[\sigma_{\ell}]-2|\to 0
$$
and then, recalling the asymptotic behavior of $\sigma_\ell$ in Proposition \ref{Conv0} we obtain
\begin{equation}
\label{dWsigma}
\dW(\sigma_\ell, 2)=
\begin{cases}
O(\ell^{-1/2}) & d=2,3 \\
O(\ell^{-3/4}) & d=4 \\
O(\ell^{-1}) & d\geq 5
\end{cases}
\end{equation}
Since $G=Z\sim \mathcal N(0,1)$, $DG=1$, that gives $\sigma_G=1$, $D^ kG=0$ for every $k\geq 2$ and $LG=-G$, so that (see \eqref{mall0}-\eqref{mall}) $\mathcal{Q}_q(G) = \mathcal{Q}_1(G)<\infty$ for every $q\geq 1$. As for $\mathcal{C}_{q,1}(\X_{\ell})$, we 
have
\begin{align*}
\mathcal{C}_{q,1}(\X_\ell)&=\|\mathcal{C}_q(\X_\ell)\|_1=\E[\big(|\X_\ell|_{1,q+1}+|L\X_\ell|_{q}\big)^ q\big(1+|\X_\ell|_{1,q+1}\big)^ {4 q}]\\
&\leq \E[\big(|\X_\ell|_{1,q+1}+|L\X_\ell|_{q}\big)^{2q}]^{\frac{1}{2}}\E[\big(1+|\X_\ell|_{1,q+1}\big)^ {8 q}]^{\frac{1}{2}}\\
&\leq (\E[|\X_\ell|_{1,q+1}^{2q}]^{\frac{1}{2}}+\E[|L\X_\ell|^{2q}]^{\frac{1}{2}})(1+\E[|\X_\ell|_{1,q+1}^{8q}]^{\frac{1}{2}})
\end{align*}
and Proposition \ref{UnifLim} allows one to check that $\sup_{\ell \mathrm{\ even} }\mathcal{C}_{q,1}(\X_\ell)<\infty$ for every $q\geq 1$. 
Then, Theorem \ref{mainThm} ensures that, for $\varepsilon>0$, 
$$
\dTV(\X_\ell, Z)\leq \textrm{C}_\varepsilon\big(\dW(\X_\ell, Z)+\dW(\det \sigma_{\ell}, 2)\big)^{1-\varepsilon}.
$$
Now, combining the above estimate on $\dW(\sigma_\ell, 2)$ and the result on $\dW(\X_\ell, Z)$ in Theorem \ref{ROS}, we conclude the proof. 

\end{proof}

Comparing \eqref{dW-th} and \eqref{dWsigma}, when applying Proposition \ref{Main} the presence of $\dW(\sigma_\ell, 2)$  does not worsen the quantitative convergence rate for $\dTV(\X_\ell, Z)$: in fact, whenever $d\geq 2$ we obtain that $\dTV(\X_\ell, Z)= O_\varepsilon(\dW(\X_\ell, Z)^ {1-\varepsilon}) $, for any $\varepsilon>0$ close to 0. In other words, the term coming from the Malliavin covariance does not slow down the convergence speed. 
%
%
%
%
%
%
%

\section{Background on Gaussian random fields}\label{background}

In this Section we recall the isonormal representation for random hyperspherical harmonics along with the Wiener-It\^o chaos theory, finally we deal with Malliavin calculus for Gaussian fields. 

\subsection{Isonormal representation}
 
Let us recall an equivalent way to introduce isonormal Gaussian random fields.
We denote $H=L^2(\mathbb{S}^d,\B(\mathbb{S}^d), \Leb)$ the real separable Hilbert space of square integrable functions on $\mathbb{S}^d$ w.r.t. the Lebesgue measure, with inner product
$
\<f,g\>_H= \int_{\mathbb{S}^d} f(x) g(x) dx.
$
We define a Gaussian white noise on $\SSd$, that is a centered Gaussian family
$
W=\{W(A), A\in \B(\SSd), \Leb(A)<+\infty\}
$
such that for $A,B \in \B(\SSd)$, we have
$$
\E[W(A)W(B)]=\int_{\SSd} 1_{A\cap B} (x) dx.
$$
Once the noise $W$ has been introduced, an isonormal Gaussian random field $T$ on $H$ can be defined as follows: for $f\in H$,
\begin{equation}\label{gausField}
T(f)=\int_{\SSd} f(x) W(dx)
\end{equation}
that is the Wiener-Ito integral of $f$ with respect to $W$. It is well known that $T$ is an isonormal Gaussian field on $H$, whose covariance is given by
\begin{equation}\label{CovTf}
\Cov(T(f), T(g))=\<f,g\>_H.
\end{equation}
Following the above construction,  one can prove that the Gaussian field $T_\ell$ in \eqref{Tl} can be represented as
\begin{equation}\label{spR}
T_\ell(x)=\int_{\mathbb{S}^d} f_{x;\ell}(y) W(dy), \quad x\in \mathbb{S}^d,
\end{equation}
in which, for $x\in \mathbb{S}^d$ and $\ell\in \N$,
\begin{equation}\label{kernel}
f_{x;\ell}(y) = \sqrt{\frac{n_{\ell;d}}{\mu_d}} G_{\ell; d}(\<x,y\>), \quad y\in\SSd.
\end{equation} 
 We also recall that $\<x,y\>=\cos d(x,y)$, where  $d(x,y)$ denotes the spherical distance between $x,y \in \mathbb S^d$, cf. (\ref{covT}).
The representation in \eqref{spR} is known to hold in law, but this is enough for our purposes. Details on this representation can be found in \cite[Chapter 2]{NP12}.

Using \eqref{CovTf}, the covariance function of the random field $T_\ell$ is given by
\begin{equation}\label{covTl}
\E[T_\ell(x) T_\ell(y)] = G_{\ell; d}(\<x,y\>),\quad  x,y\in \SSd,
\end{equation}
in fact \eqref{covTl} is an immediate consequence of the following reproducing property: for every $x,y\in\SSd$,
\begin{equation}\label{G1}
\int_{\SSd} G_{\ell;d}(\<x,z\>)G_{\ell;d}(\<z,y\>)dz=\frac{\mu_d}{n_{\ell;d}}G_{\ell; d}(\<x,y\>).
\end{equation}
For later use, we recall the asymptotic behavior as $\ell\to\infty$ of the moments of the Gegenbauer polynomials, that we resume as follows  (proof details on the constants can be found in \cite[Proposition 1.1]{ROS20}).

\begin{proposition}\label{GegProp}
	
For $q\in\N$, $q\geq 2$, set
\begin{equation}\label{cqd}
c_{q;d}= 
\begin{cases}
\Big(2^ {\frac d2 -1}\Big(\frac d2-1\Big)! \Big)^ q\int_0^ \infty J_{\frac d2 -1}(u)^ q u^ {-q(\frac d2-1)+d-1}du& \mbox{ if } q\geq 3,\cr
\frac{(d-1)! \mu_d}{4\mu_{d-1}}&\mbox{                                     if } q=2,\\
\end{cases}
\end{equation}
where $J_{\frac d2 -1}$ is the Bessel function of order $\frac d2 -1$. 

For $q\geq 2$ and $d\geq 2$, the function $\SSd\ni y \mapsto \int_{\SSd} G_{\ell; d}(\<x,y\>)^q dx$ is constant. Moreover, the following properties hold.

For $d\geq 2$ one has $\int_{\SSd} G_{\ell; d}(\<x,y\>)^2 dx =\frac{\mu_d}{ n_{\ell;d} }$ and, as  $\ell\to\infty$,
\begin{equation}\label{asGeg2}
\int_{\SSd} G_{\ell; d}(\<x,y\>)^2 dx =
2\mu_{d-1} \frac{c_{2;d}}{\ell^{d-1}}(1+o_{2;d}(1)).
\end{equation}

Set now $q\geq 3$.  Then 
\begin{itemize}
	\item if $d\geq 3$, then
\begin{equation}\label{asGegq}
\int_\SSd G_{\ell;d} (\<x,y\>)^q dx = 2 \mu_{d-1}\frac{c_{q;d}}{\ell^d} (1+o_{q;d}(1));
\end{equation}

\item
if $d=2$, the behavior differs according to $q\neq 4$ (being as in \eqref{asGegq}) and $q= 4$: 
\begin{equation}\label{asLegq}
\int_\SS2 G_{\ell; 2}(\<x,y\>)^ qdx\equiv 
\int_\SS2
P_\ell(\<x,y\>)^q dx = 
\begin{cases}
\frac{12\log \ell}{\pi\ell^2}(1+o_{4;2}(1)) & q=4\\
\frac{4\pi c_{q;2}}{\ell^2}(1+o_{q;2}(1)) & q=3\mbox{ or } q\geq 5.
\end{cases}
\end{equation}

\end{itemize}

\end{proposition}

\subsection{Wiener chaos expansion}

Let us recall the notion of Wiener chaos.
Let $H_q(x)$ denote the Hermite polynomial of degree $q\in\N$, defined by
$$
H_q(x)=
(-1)^q e^{\frac{x^2}{2}}\frac{d^q}{dx^q}(e^{-\frac{x^2}{2}}),\quad q\geq 1
$$
and $H_0(x)=1$ (\cite[\S 5.5]{Sze39}). The family of Hermite polynomials $\left(H_q \right)_{q\geq 0}$ is an orthogonal basis of $L^2(\R, \B(\R), \nu)$. 
For jointly Gaussian random variables $Z_1, Z_2 \sim \mathcal N(0,1)$, and $q_1,q_2\geq 0$ we have
\begin{equation}\label{hermiteCov}
\E[H_{q_1}(Z_1)H_{q_2}(Z_2)] = q_1 !\, \E[Z_1 Z_2]^{q_1} \1_{q_1=q_2}.
\end{equation}

Let $T$ denote a Gaussian random field generated by the Gaussian noise $W$ as in \eqref{gausField} and let $\F_T$ denote the $\sigma$-algebra generated by $T$: $\F_T=\sigma(T(f)\,:\,f\in H)$, where $H=L^2(\SSd, \B(\SSd), \Leb)$.
For any $q\in \N$ we consider the closure $\C_q$ in $L^2(\Omega,\F_T,\P)$ of the linear space generated by 
$$
\{H_q( T(f) )\,:\, f\in H \mbox{ with } \|f\|_{H}= 1\}.
$$
The space $\C_q$ is called the $q$-th Wiener chaos associated with $T$. From \eqref{hermiteCov}, it immediately follows that $\C_q \perp \C_{q'}$ when $q\neq q'$.
Then the following chaotic Wiener-Ito expansion holds (see e.g. \cite{NP12}): 
$$
L^2(\Omega, \F_T, \P)=\bigoplus_{q=0}^{+\infty} \C_q.
$$
This means that every $F\in L^2(\Omega, \F_T,\P)$ can be uniquely decomposed in $L^2(\P)$ as 
$$
F=\sum_{q\geq 0} J_q(F)
$$
where $J_q$ is the orthogonal projection operator on $\C_q$. 

Once the Wiener chaos expansion is defined, we can introduce the Ornstein-Uhlenbeck operator $L$, which will play an important role in our approach: for $F\in L^ 2(\Omega, \F_T,\P)$, one says that $F\in Dom(L)$ if and only if
$$
\sum_{q \geq 1} q^2\E[J_q(F)^ 2]<\infty
$$
and in such a case,
\begin{equation}\label{defL}
LF = -\sum_{q \geq 1} q J_q(F).
\end{equation}

\subsection{Malliavin calculus for Gaussian random fields}\label{sect:malliavin}

In this section we recall some definitions and properties of Malliavin calculus for Gaussian random fields --  details and results can be found fully explained in \cite[\S 2.3]{NP12} or \cite{Nua}. For $k\in\N$, we denote with $C^ k_p(\R^ m)$, respectively $C^ k_b(\R^ m)$, the set of functions $f:\R^ m\to\R$ that are continuously differentiable up to order $k$ and with polynomial growth, respectively bounded. And as usual,  $C^ \infty_p(\R^ m)=\cap_{k\geq 0}C^ k_p(\R^ m)$ and $C^ \infty _b(\R^ m)=\cap_{k\geq 0}C^ k_b(\R^ m)$.

Let $T$ be the isonormal Gaussian random field based on the Gaussian noise $W$ as in \eqref{gausField}. Let $\mathcal{S}$ denote the set of the simple functionals, being defined as the variables of the form 
\begin{equation}\label{smoothF}
f(T(g_1),\ldots, T(g_m))
\end{equation}
where $m\geq 1$, $f\in C^ \infty_p(\R^m)$ and $g_1,\ldots, g_m\in H$. 
It is well known that $\mathcal{S}$ is dense in $L^p(\Omega):=L^p(\Omega, \F_T, \P )$.

Given $k\in \N$, we denote with $H^{\otimes k}$ and $H^{\odot k}$, respectively, the $k$-th tensor product and the $k$-th symmetric tensor product of $H=L^ 2(\SSd,\B(\SSd), \Leb)$.
Let $F\in \mathcal{S}$ be given by \eqref{smoothF} and $k \in \N$. The $k$-th \textit{Malliavin derivative}  is the element of $L^2(\Omega; H^{\odot k})$ defined by
$$
D^{(k)} F=\sum_{i_1,\ldots, i_k=1}^m \frac{\partial^k f}{\partial x_{i_1}\cdots \partial x_{i_k}}(T(g_1),\ldots, T(g_m))\, g_{i_1} \otimes \cdots \otimes g_{i_m}.
$$
We recall that $D^{(k)}: \mathcal{S} \subset L^p(\Omega;\R)\hookrightarrow L^p(\Omega;H^{\odot k})$ is closable.
So, for $k\in \N$ and $p\geq 1$, one defines the space $\DD^{k,p}$ as the closure of $\mathcal{S}$ with respect to the norm 
$$
\|F \|_{\DD^{k,p}} =\big(\E[|F|^p]+\E[\|F\|_{H}^p]+\ldots +\E[\|D^{(k)} F\|_{H^{\otimes k}}^p]\big)^{\frac{1}{p}}
$$
and the Malliavin derivative can be extended to the set $\DD^{k,p}$, being the \textit{domain} of $D^{(k)}$ in $L^p(\Omega;\R)$. 
In particular, the space $\DD^{k,2}$ is a Hilbert space with respect to the inner product
$$
\<F,G\>_{\DD^{k,2}}=\E[FG]+\sum_{r=1}^k \E[\<D^{(r)}F,D^{(r)}G\>_{H^{\odot r}}].
$$
Moreover, the chain rule property does hold: for every $\phi\in C^ 1_b(\R^ m)$ and $F=(F_1,\ldots, F_m)$ with $F_i\in \DD^{1,p}$, $i=1,\ldots,m$, for some $p\geq 1$, then $\phi(F)\in\DD^{1,p}$ and 
\begin{equation}\label{chainRule}
D\phi(F)=\sum_{r=1}^k \frac{\partial \phi}{\partial x_r}(F) D F_r.
\end{equation}

It is well known that such a Malliavin calculus framework satisfies the abstract hypotheses required in \cite[Section 2.1]{BCP19} and resumed here in \S \ref{sect:proofs} (see e.g. \cite{Nua} or \cite{NP12}): just take $\mathcal{E}=\mathcal{S}$, 
	$\mathcal{H}=H=L^ 2(\SSd,\B(\SSd), \Leb)$ and $L$ as the Ornstein-Uhlenbeck operator defined in \eqref{defL}. In particular, the duality relationship \textbf{(M3)} does hold for $F,G\in\DD^ {2,2}$ and therefore, it holds true on $\mathcal{E}$.

To conclude, we give some formulas that will be used in the sequel. 
Let $x\in \SSd$ and $T_\ell(x)$ be the Gaussian random field in \eqref{spR}. As an immediate consequence of the chain rule \eqref{chainRule}, for  $q\in \N$ and $p\geq 1$ then $H_q(T_\ell(x))\in \DD^{1,p}$ and, from \eqref{smoothF} and \eqref{spR},
$$
D_y H_q(T_{\ell}(x))=q H_{q-1}(T_{\ell}(x)) D_yT_\ell(x)
=q H_{q-1}(T_{\ell}(x)) \sqrt{\frac{n_{\ell;d}}{\mu_d}} G_{\ell;d}(\<x, y \>).
$$
Iterating the argument, $H_q(T_\ell(x))\in \DD^{k,p}$ for every $k\in\N$ and
\begin{equation}\label{DkHp}
D^{(k)}_{y_1,\ldots, y_k } H_q(T_\ell (x)) =  \Big(\frac{n_{\ell;d}}{\mu_d}\Big)^{\frac{k}{2}}\frac{q!}{(q-k)!} H_{q-k}(T_\ell(x)) \prod_{r=1}^k G_{\ell;d}(\<x, y_r\>).
\end{equation}
Moreover, by developing standard density arguments, \eqref{DkHp} gives
 $\int_{\SSd} H_q(T_\ell (x)) dx\in \DD^{k,p}$ 
and
\begin{equation}\label{DkIntHp}
D^{(k)}_{y_1,\ldots, y_k } \int_\SSd H_q(T_\ell (x)) dx=  
\Big(\frac{n_{\ell;d}}{\mu_d}\Big)^{\frac{k}{2}}\frac{q!}{(q-k)!} \int_{\SSd}H_{q-k}(T_\ell(x)) \prod_{r=1}^k G_{\ell;d}(\<x, y_r\>) dx.
\end{equation}

\section{Convergence of Malliavin covariances}\label{sect:Conv0} 

In this section we prove Lemma \ref{Conv0}.
Recalling the (finite dimensional) chaos expansion for $\varphi(Z)$ in \eqref{chphi} and substituting it in \eqref{Xl}, we obtain the following expansion for $X_\ell$:
\begin{equation}\label{chaosXl}
X_\ell=m_{\ell;d}+\sum_{q \geq 2} \frac{b_q}{q!} \int_{\mathbb S^d}  H_q(T_\ell(x)) dx,
\end{equation}
where
$$
m_{\ell;d}=\E[X_\ell]=\E[\varphi(Z)] \mu_d .
$$
Equation \eqref{chaosXl} is really the chaos expansion of $X_\ell$. One has in fact, for every $f\in H$ and $p\neq q$, 
$
\E[ H_{p}(T(f))\int_{\mathbb S^d}  H_q(T_\ell(x)) dx]
=\int_{\mathbb S^d} \E[ H_{p}(T(f)) H_q(T_\ell(x))]dx=0.
$

Notice that \eqref{chaosXl} says that the projection on the chaos of order 1 is null, as briefly mentioned above. In fact, recalling \eqref{chphi}, by  \eqref{Tl}
$$
J_1(X_\ell)=b_1\int_{\SSd}T_\ell(x)dx=b_1\sqrt{\frac{\mu_d}{n_{\ell;d}}}\sum_{m=1}^{n_{\ell;d}} a_{\ell,m} \int_{\SSd} Y_{\ell,m;d}(x) dx=0.
$$
Following \eqref{chaosXl}, the chaos expansion of the normalized r.v. $\X_\ell$ is given by
\begin{equation}\label{chaos-tildeX}
\X_\ell=\frac 1{v_{\ell;d}}\,\sum_{q \geq 2} \frac{b_q}{q!} \int_{\mathbb S^d}  H_q(T_\ell(x)) dx,\mbox{ where }v^2_{\ell;d}=\Var(X_{\ell}).
\end{equation} 
Notice that, by \eqref{varXl} from Theorem \ref{ROS} and \eqref{n-ell}, 
\begin{equation}\label{v-ell}
v_{\ell;d}\sim b_2c_d \ell^{- \frac{d-1}{2}}.
\end{equation}

\subsection{Diagram formula and graphs}

Now we introduce some notation and results that are useful to prove Lemma \ref{Conv0}. We first study a different way to write down the well known diagram formula \cite[Proposition 4.15]{MPbook}. To this purpose, we need to introduce a set that we will use several times.

\begin{definition}\label{defA}
For $q_1,\ldots,q_n\in\N$, we define $\mathcal{A}_{q_1,\ldots,q_n}$ as the set given by the indexes $\{k_{ij}\}^n_{{ i,j=1}}$ such that for every $i,j=1,\ldots,n$, 
\begin{equation}\label{A}
\mbox{$k_{i,j}\in\N$, $k_{ii}=0$, $k_{ij}=k_{ji}$ \text{ and } $\sum_{j=1}^n k_{ij}=q_i$.}
\end{equation}
\end{definition}

\begin{lemma}\label{DIAGRAMMA}
Let $n\geq 2$ and let $(Z_1,\ldots, Z_n)$ be a $n$-dimensional centered Gaussian vector. For $q_1,\ldots,q_n\in\N$, consider  $\mathcal{A}_{q_1,\ldots,q_n}$ as in Definition \ref{defA}.  Then, 
\begin{equation}\label{eqDIAGRAMMA}
\E[\prod_{r=1}^n H_{q_r}(Z_r)]=\prod_{r=1}^n q_r! \times \sum_{\{k_{i,j}\}^n_{{ i,j=1}} \in \mathcal{A}_{q_1,\ldots,q_n}  } \prod_{\underset{i<j}{i,j=1}}^n\frac{\E[Z_iZ_j]^{k_{ij}}}{k_{ij}!}. 
\end{equation}
In particular, taking $Z_1=\cdots=Z_n=Z\sim \mathcal N(0,1)$, one has
\begin{equation}\label{PassDel}
\E[\prod_{r=1}^n H_{q_r}(Z)]=\prod_{r=1}^n q_r! \times \sum_{\{k_{i,j}\}^n_{{ i,j=1}} \in \mathcal{A}_{q_1,\ldots,q_n}  } \prod_{\underset{i<j}{i,j=1}}^n\frac{1}{k_{ij}!}. 
\end{equation}

\end{lemma}

The proof of Lemma \ref{DIAGRAMMA} is postponed to Appendix \ref{Diag}. Let us see how we apply such result.

For fixed $n\geq 2$ and $x_1,\ldots,x_n \in \SSd $ , the random vector $(T_\ell(x_1),\ldots , T_\ell( x_n ))$ is a centered Gaussian random vector whose covariances are given by (see \eqref{covTl})
$$
\E[T_\ell(x_i)T_\ell(x_j)]=G_{\ell;d}(\<x_i,x_j\>), \quad i,j=1,\ldots,n.
$$
When dealing with our proofs, we often need to compute and/or estimate quantities of the type 
$$
\int_{(\SSd)^ n}\E[\prod_{r=1}^n H_{q_r}(T_\ell(x_r))]dx.
$$
By using \eqref{eqDIAGRAMMA}, we have
\begin{equation}\label{eqDIAGRAMMA2}
\int_{(\SSd)^ n}\E[\prod_{r=1}^n H_{q_r}(T_\ell(x_r))]dx=\prod_{r=1}^n q_r! \times \sum_{\{k_{i,j}\}^n_{{ i,j=1}} \in \mathcal{A}_{q_1,\ldots,q_n}  } \int_{(\SSd)^ n}\prod_{\underset{i<j}{i,j=1}}^n\frac{G_{\ell;d}(\<x_i,x_j\>)^{k_{ij}}}{k_{ij}!} dx. 
\end{equation}
Therefore, it would be very useful to get a good estimate for the integrals appearing in the r.h.s of \eqref{eqDIAGRAMMA2}, that is, when the integrand function is the product of a number of powers of Gegenbauer polynomials. 
To this purpose we need to introduce the concept of \textit{extrapolated graph} from  a given $\kappa=\{k_{ij}\}^n_{{ i,j=1}}\in \mathcal{A}_{q_1,\ldots,q_n}$. Such graph is defined as the pair $\mathfrak{G}_\kappa=(V,E_\kappa)$ in which the set of the nodes is given by $V=\{1,\ldots,n\}$ and the set of the edges is given as follows: the edge $(i,j)$ does exist iff $k_{ij}\neq 0$ (notice that, since $k_{ii}=0$, there are no self-loops).  The use of graphs is a key point in our approach, that is why in Appendix \ref{sect:graph} we recall the main definitions and properties.

\begin{lemma}\label{STIMA}
For $n\in \N^*$, let $\kappa=\{k_{ij}\}^{n}_{{ i,j=1}}\in \mathcal{A}_{q_1,\ldots,q_{n}}$ be fixed. Let $\mathfrak{G}_\kappa$ denote the extrapolated graph from $\kappa$ and $N_\kappa$ denote the number of connected components of $\mathfrak{G}_\kappa$. Then,
\begin{equation}\label{stimaGRAFI}
\int_{(\SSd)^{n}} \prod_{\underset{i<j}{i,j= 1}}^{n} G_{\ell;d}(\<x_i,x_j\>)^{2k_{ij}} dx \leq \frac{C_{d}(N_\kappa) }{\ell^{(d-1)(n-N_\kappa) }}
\end{equation}
where $C_{d}(N_\kappa)=(8 \mu_d \mu_{d-1} c_{2;d})^{n-N_\kappa} \mu_d^{N_\kappa}$, $c_{2;d}$ being given in \eqref{cqd}. As a consequence, for $n=2p$, 
\begin{equation}\label{connessiSTIMA}
\int_{(\SSd)^{2p}} \prod_{\underset{i<j}{i,j= 1}}^{2p} G_{\ell;d}(\<x_i,x_j\>)^{2k_{ij}} dx\leq \frac{C_{d;p}}{\ell^{(d-1)p}},
\end{equation}
where  $C_{d;p}=(2 (d-1)!\mu_d^2)^{2p} \mu_d^p$. 
\end{lemma}
The proof of Lemma \ref{STIMA} can be found in Appendix \ref{sect:graph}. 
Let us remark that, in principle, (\ref{connessiSTIMA}) might be useful to get some estimates on concatenated sums of products of so-called Clebsch-Gordan coefficients $\lbrace C^{L,M}_{\ell_1 ,m_1, \ell_2, m_2}\rbrace$ that we define by 
\begin{equation*}\label{CG}
Y_{\ell_1, m_1;d}(x) Y_{\ell_2,m_2;d}(x) = \sum_{L=0}^{\ell_1+\ell_2}\sum_{M=1}^{n_{L;d}} C^{L,M}_{\ell_1, m_1, \ell_2, m_2;d} Y_{L,M;d}(x),\quad x\in 	\mathbb S^d.
\end{equation*}
This is because there exists a precise link between such quantities and moments of Gegenbauer polynomials which can be established via the addition formula \eqref{covT}. 
However, it is not clear whether it is actually possible to obtain optimal or novel estimates, even if in dimension $ d> 2 $ a little is known about these coefficients. (See \cite[Section 3.5]{MPbook} for a complete discussion in the case of the $2$-sphere).

\subsection{Proof of Theorem \ref{Conv0}}\label{3.2}

We are now in a position to prove Theorem \ref{Conv0}, that is the main result on the convergence in Wasserstein distance for the Malliavin covariances of $\tilde X_\ell$, as $\ell\to +\infty$.  Let us anticipate that 
the proof requires a finer different method for the case $d=2$ than $d\ge 3$. Therefore, as it will be clear from reading the proof, we will have to split in two different approaches.

\smallskip

\begin{proof}[Proof of Theorem \ref{Conv0}]
By using \eqref{DkIntHp} (with $k=1$) and classical density arguments, the Malliavin derivative $D\tilde X_\ell: \Omega \to H$ is given by
$$
D_y \X_\ell
=\frac{1}{v_{\ell;d}} \sqrt{\frac{n_{\ell;d}}{\mu_d}} \sum_{q\geq 2} \frac{b_q}{(q-1)!}\int_{\mathbb{S}^d} H_{q-1}(T_\ell(x))  G_{\ell; d}(\<x,y\>) dx.
$$
Following \eqref{MallCov}, with $n=1$ and $\mathcal{H}=H=L^ 2(\SSd,\B(\SSd), \Leb)$, we can write down  the Malliavin covariance $\sigma_\ell$ of $\X_\ell$:
\begin{align*}
\sigma_\ell
=&\int_{\mathbb{S}^d} |D_y \X_\ell|^2 dy= \frac{1}{v_{\ell;d}^2}\frac{n_{\ell;d}}{\mu_d}\sum_{q, p\geq 2} \frac{b_q b_p}{(q-1)!(p-1)!}\times \\
&\times  \int_{\mathbb{S}^d} \int_{(\mathbb{S}^d)^2}H_{q-1}(T_\ell(x))  H_{p-1}(T_\ell(z))  G_{\ell; d}(\<x,y\>)G_{\ell; d}(\<z,y\>) dx dz dy.
\end{align*}
By using the duplication formula \eqref{G1}, we obtain

\begin{equation}\label{sigma}
\sigma_\ell
=\frac{1}{v_{\ell;d}^2}\sum_{q, p\geq 2} \frac{b_q b_p}{(q-1)!(p-1)!}  \int_{(\mathbb{S}^d)^2} H_{q-1}(T_\ell(x))  H_{p-1}(T_\ell(z)) G_{\ell; d}(\<x,z\>) dxdz.
\end{equation}
Therefore, by \eqref{covTl},
\begin{align*}
\E[\sigma_\ell]&=\frac{1}{v_{\ell;d}^2}\sum_{q, p\geq 2} \frac{b_q b_p}{(q-1)!(p-1)!}  \int_{(\mathbb{S}^d)^2} \E[H_{q-1}(T_\ell(x))  H_{p-1}(T_\ell(z)) ] G_{\ell; d}(\<x,z\>) dxdz\\
&=\frac{1}{v_{\ell;d}^2}\sum_{q\geq 2} \frac{b_q^2}{(q-1)!}  \int_{(\mathbb{S}^d)^2}  G_{\ell; d}(\<x,z\>)^q dxdz.
\end{align*}
Then, from the asymptotics for moment of Gegenbauer polynomials in Proposition \ref{GegProp} and from \eqref{v-ell}, we have that 
$$
\E[\sigma_\ell]-2= O\left(  \1_{d\geq 3} \frac{1}{\ell}+ \1_{d=2}\left( \1_{b_4 \neq 0}\frac{\log \ell}{\ell}+ \1_{b_4=0} \frac{1}{\ell}      \right )\right)
$$ as $\ell \to \infty$. In the above result we underline that the difference between $d=2$ and $d\geq 3$ changes the asymptotic behavior when $b_4\neq 0$.

Now we study the variance of $\sigma_\ell$.
Denoting with $dx:=dx_1 dx_2 dx_3 dx_4$, we have that 
\begin{align*}
&\E[\sigma_\ell^2]=\\
&\frac{1}{v_{\ell;d}^4} \sum_{q_1,q_2, q_3, q_4\geq 2} \Big(\prod_{j=1}^4 \frac{b_{q_j}}{(q_j-1)!} \Big)\int_{(\mathbb S^d)^4 } \E[\prod_{i=1}^4 H_{q_i-1}(T_\ell(x_i))] G_{\ell;d}(\<x_1,x_2\>)G_{\ell;d}(\<x_3,x_4\>)dx.
\end{align*} 
By using Lemma \ref{DIAGRAMMA}, we have
\begin{align*}
&\E[\sigma_\ell^2]=\\
&\frac{1}{v_{\ell;d}^4} \sum_{q_1,q_2, q_3, q_4\geq 2} \Big(\prod_{j=1}^4 \frac{b_{q_j}}{(q_j-1)!} \Big)\prod_{r=1}^4 (q_r-1)!  \times \\
&\sum_{\{k_{i,j}\}^4_{{ i,j=1}} \in \mathcal{A}_{q_1-1,\ldots,q_4-1}  }
 \prod_{\underset{i<j}{i,j=1}}^4\frac{1}{k_{ij}!} \int_{(\mathbb S^d)^4 }\prod_{\underset{i<j}{i,j=1}}^4 G_{\ell;d}(\<x_i,x_j\>)^{k_{ij}} G_{\ell;d}(\<x_1,x_2\>)G_{\ell;d}(\<x_3,x_4\>)dx.
\end{align*} 
First of all, when $q_1=q_2$ and $q_3=q_4$ we have
\begin{align*}
&\sum_{q_1 q_3\geq 2} b_{q_1}^2 b_{q_3}^2  \sum_{\{k_{i,j}\}^4_{{ i,j=1}} \in \mathcal{A}_{q_1-1,\ldots, q_3-1}  } \prod_{\underset{i<j}{i,j=1}}^4\frac{1}{k_{ij}!}\times \\
&\times  \int_{(\mathbb S^d)^4 }\prod_{\underset{i<j}{i,j=1}}^4 G_{\ell;d}(\<x_i,x_j\>)^{k_{ij}} G_{\ell;d}(\<x_1,x_2\>)G_{\ell;d}(\<x_3,x_4\>)dx.
\end{align*}
 Now, when $k_{13}=k_{14}=k_{23}=k_{24}=0$, one gets $k_{12}=q_1-1$ and $k_{34}=q_3-1$. So it remains
 \begin{align*}
&\sum_{q_1 q_3\geq 2} \frac{b_{q_1}^2 b_{q_3}^2}{(q_1-1)!(q_3-1)!}\int_{(\mathbb S^d)^4 } G_{\ell;d}(\<x_1,x_2\>)^{q_1}G_{\ell;d}(\<x_3,x_4\>)^{q_3}dx.
\end{align*}  
This is exactly $v_{\ell;d}^4\, \E[\sigma_\ell]^2$.
Then we define 
$\mathcal{N}_{q_1-1, q_2-1, q_3-1, q_4-1}$ as the set of $\kappa=\{k_{ij}\}_{i,j=1}^ 4$ such that
$$
\mbox{ $k_{12}=q_1-1=q_2-1, \,\,\,  k_{34}=q_3-1=q_4-1,\,\,\, k_{13}=k_{14}=k_{23}=k_{24}=0$ }
$$
and we set
$$
\mathcal{C}_{q_1-1,q_2-1,q_3-1, q_4-1}=\mathcal{A}_{q_1-1,\ldots,q_4-1}\setminus \mathcal{N}_{q_1-1, q_2-1, q_3-1, q_4-1}.
$$
Hence
\begin{align*}
&\Var(\sigma_\ell )=\E[\sigma_\ell^ 2]-\E[\sigma_\ell]^ 2
=\frac{1}{v_{\ell;d}^4} \sum_{q_1,q_2, q_3, q_4\geq 2} \prod_{i=1}^ 4\frac{b_{q_i}}{(q_i-1)!} \prod_{r=1}^4 (q_r-1)! \times\\
&\times  \sum_{\{k_{i,j}\}^4_{{ i,j=1}} \in \mathcal{C}_{q_1-1,\ldots,q_4-1}  } \prod_{\underset{i<j}{i,j=1}}^4\frac{1}{k_{ij}!} \int_{(\mathbb S^d)^4 }\prod_{\underset{i<j}{i,j=1}}^4 G_{\ell;d}(\<x_i,x_j\>)^{k_{ij}} G_{\ell;d}(\<x_1,x_2\>)G_{\ell;d}(\<x_3,x_4\>)dx,
\end{align*}
so that
\begin{align}
&\Var(\sigma_\ell )
\leq \frac{1}{v_{\ell;d}^4} \sum_{q_1,q_2, q_3, q_4\geq 2} \prod_{i=1}^ 4\frac{|b_{q_i}|}{(q_i-1)!} \prod_{r=1}^4 (q_r-1)! \times\nonumber\\
&\times  \sum_{\{k_{i,j}\}^4_{{ i,j=1}} \in \mathcal{C}_{q_1-1,\ldots,q_4-1}  } \prod_{\underset{i<j}{i,j=1}}^4\frac{1}{k_{ij}!} \Big|\int_{(\mathbb S^d)^4 }\prod_{\underset{i<j}{i,j=1}}^4 G_{\ell;d}(\<x_i,x_j\>)^{k_{ij}} G_{\ell;d}(\<x_1,x_2\>)G_{\ell;d}(\<x_3,x_4\>)dx\Big|\label{appoggio}
\end{align}
We now prove that there exists $c>0$ such that for every $\{k_{i,j}\}^n_{{ i,j=1}} \in \mathcal{C}_{q_1-1,\ldots,q_4-1} $,
\begin{equation}\label{app1}
\Big|\int_{(\SSd)^4} \prod_{\underset{i<j}{i,j=1}}^4 G_{\ell;d}(\<x_i,x_j\>)^{k_{ij}} G_{\ell;d}(\<x_1,x_2\>)G_{\ell;d}(\<x_3,x_4\>)dx\Big|
\leq \frac{c_d}{\ell^{2d-2+\frac{d-1}{2}}}.
\end{equation}  
\eqref{app1} will follow by applying Lemma \ref{STIMA}. For a fixed  $\kappa=\{k_{ij}\}_{i,j=1}^ 4\in  \mathcal{C}_{q_1-1,\ldots,q_4-1} $,let $N_\kappa$ be the number of the connected components of the extrapolated graph $\mathfrak{G}_\kappa$. We observe that $N_\kappa\in\{1,2\}$, recall in fact that by \eqref{defA} for any $i$ there exists at least an index $j\neq i$ such that $k_{ij}>0$ (here, $q_i-1\geq 1$ for every $i$). So, we split our reasoning according to $N_\kappa=1$ and $N_\kappa=2$.

\smallskip

\textbf{Case 1: $N_\kappa=1$.} By using the Cauchy-Schwarz inequality we have
\begin{align*}
&\Big|\int_{(\SSd)^4} \prod_{\underset{i<j}{i,j=1}}^4 G_{\ell;d}(\<x_i,x_j\>)^{k_{ij}} G_{\ell;d}(\<x_1,x_2\>)G_{\ell;d}(\<x_3,x_4\>)dx\Big|\leq \\
&\Big(\int_{(\SSd)^4} \prod_{\underset{i<j}{i,j=1}}^4 G_{\ell;d}(\<x_i,x_j\>)^{2k_{ij}} dx\Big)^ {1/2}
\Big(\int_{(\SSd)^4} G_{\ell;d}(\<x_1,x_2\>)^2G_{\ell;d}(\<x_3,x_4\>)^2dx\Big)^{1/2}
\end{align*}  
Estimating the first factor by means of \eqref{stimaGRAFI} with $N_\kappa=1$ and computing the second factor by means of \eqref{asGeg2}, straightforward computations give \eqref{app1}

\smallskip

\textbf{Case 2: $N_\kappa=2$.}  Figure \ref{N=2} shows all possible extrapolated graphs having $2$ connected components. We 
notice that the graph in (a) is extrapolated by an element $\kappa=\{k_{ij}\}_{i,j=1}^ 4$ belonging to $\mathcal{N}_{q_1-1,\ldots,q_4-1} $. Such indexes have been already deleted, so we study the cases shown in (b) and in (c).
\begin{center}
	\begin{figure}[h]
		{\footnotesize 	\begin{center}
				\begin{tikzpicture}
				\node[shape=circle,draw](A) at (1,3) {$1$};
				\node[shape=circle,draw](B) at (3,3) {$2$};
				\node[shape=circle,draw](C) at (3,1) {$3$};
				\node[shape=circle,draw](D) at (1,1) {$4$};
				\draw (A) -- (B);
				\draw (C) -- (D);
				\draw (2,0) node{(a)};
				\node[shape=circle,draw](A) at (4.5,3) {$1$};
				\node[shape=circle,draw](B) at (6.5,3) {$2$};
				\node[shape=circle,draw](C) at (6.5,1) {$3$};
				\node[shape=circle,draw](D) at (4.5,1) {$4$};
 				\draw (A) -- (C);
                \draw (B) -- (D);
 
				\draw (5.5,0) node{(b)};
				\node[shape=circle,draw](A) at (8,3) {$1$};
				\node[shape=circle,draw](B) at (10,3) {$2$};
				\node[shape=circle,draw](C) at (10,1) {$3$};
				\node[shape=circle,draw](D) at (8,1) {$4$};
				\draw (A) -- (D);
                \draw (B) -- (C);
				\draw (9,0) node{(c)};
				\end{tikzpicture}
						\end{center}
		}
		\caption{\small All possible extrapolated graphs $\mathfrak{G}_\kappa$ from $\kappa=\{k_{ij}\}_{i,j=1}^4$  having exactly two connected components. }\label{N=2}
	\end{figure}
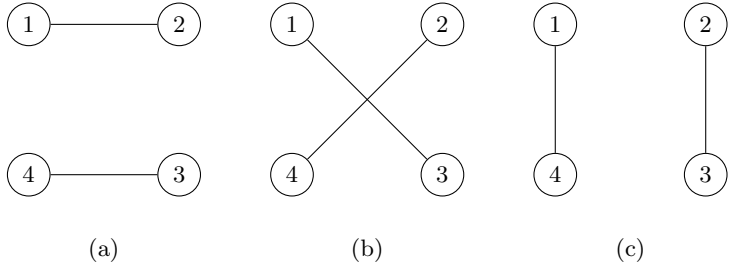
\end{center}
As for case (b), we have
\begin{align*}
&\Big|\int_{(\SSd)^4} \prod_{\underset{i<j}{i,j=1}}^4 G_{\ell;d}(\<x_i,x_j\>)^{k_{ij}} G_{\ell;d}(\<x_1,x_2\>)G_{\ell;d}(\<x_3,x_4\>)dx\Big|\\
&=\Big|\int_{(\SSd)^4} G_{\ell;d}(\<x_1,x_2\>)G_{\ell;d}(\<x_1,x_3\>)^{q_1-1}G_{\ell;d}(\<x_2,x_4\>)^{q_2-1} G_{\ell;d}(\<x_3,x_4\>) dx\Big|.
\end{align*}
Assume that $q_1=2$ or $q_2=2$. W.l.g we set $q_1=2$. Then, using \eqref{G1}, we have
\begin{align*}
&\Big|\int_{(\SSd)^4} G_{\ell;d}(\<x_1,x_2\>)G_{\ell;d}(\<x_1,x_3\>)G_{\ell;d}(\<x_2,x_4\>)^{q_2-1} G_{\ell;d}(\<x_3,x_4\>) dx_1dx_2dx_3dx_4\Big|\\
&=\Big|  \frac{\mu_d}{n_{\ell;d}} \int_{(\SSd)^3} G_{\ell;d}(\<x_2,x_3\>)G_{\ell;d}(\<x_2,x_4\>)^{q_2-1} G_{\ell;d}(\<x_3,x_4\>) dx_2dx_3dx_4 \Big|\\
&=\Big|  \frac{\mu_d^2}{n_{\ell;d}^2} \int_{(\SSd)^2}G_{\ell;d}(\<x_2,x_4\>)^{q_2}  dx_2dx_4 \Big|\\
&\leq \frac{\mu_d^2}{n_{\ell;d}^2}\int_{(\SSd)^2}G_{\ell;d}(\<x_2,x_4\>)^{2}  dx_2dx_4 \leq \frac{c_d}{\ell^{3d-3}},
\end{align*}
the last inequality following from \eqref{asGeg2}. If instead $q_1\geq 3$ and $q_2 \geq 3$,
\begin{align*}
&\Big|\int_{(\SSd)^4} G_{\ell;d}(\<x_1,x_2\>)G_{\ell;d}(\<x_1,x_3\>)^{q_1-1}G_{\ell;d}(\<x_2,x_4\>)^{q_2-1} G_{\ell;d}(\<x_3,x_4\>) dx\Big|\\
&\leq \int_{(\SSd)^4} G_{\ell;d}(\<x_1,x_3\>)^2 G_{\ell;d}(\<x_2,x_4\>)^2 |G_{\ell;d}(\<x_3,x_4\>)| dx
\end{align*}
By integrating first w.r.t $x_1$, then w.r.t. $x_2$ and by using \eqref{asGeg2}, we get
\begin{align*}
&\Big|\int_{(\SSd)^4} G_{\ell;d}(\<x_1,x_2\>)G_{\ell;d}(\<x_1,x_3\>)^{q_1-1}G_{\ell;d}(\<x_2,x_4\>)^{q_2-1} G_{\ell;d}(\<x_3,x_4\>) dx\Big|\\
&
\leq \frac{c_d}{\ell^{2d-2}} \int_{(\SSd)^2} |G_{\ell;d}(\<x_3,x_4\>)| dx_3dx_4.
\end{align*}
Now we use the Cauchy-Schwarz inequality and again apply  \eqref{asGeg2}. We finally obtain \eqref{app1}.

Case (c) in Figure \ref{N=2} can be treated analogously, so \eqref{app1} finally holds.

Coming back to the study of $\Var(\sigma_\ell)$, we use \eqref{v-ell}, we insert the estimates \eqref{app1} in \eqref{appoggio} and we have 
$$
\Var(\sigma_\ell )
\leq \frac{c_d \, \ell^{2d-2}}{\ell^{2d-2+\frac{d-1}{2}}} \sum_{q_1,q_2, q_3, q_4\geq 2} \prod_{i=1}^ 4\frac{|b_{q_i}|}{(q_i-1)!}
 \prod_{r=1}^4 (q_r-1)! \sum_{\{k_{i,j}\}^4_{{ i,j=1}} \in \mathcal{A}_{q_1-1,\ldots,q_4-1}  } \prod_{\underset{i<j}{i,j=1}}^4\frac{1}{k_{ij}!} 
$$
We now use \eqref{PassDel}: for $Z\sim \mathcal N(0,1)$,
\begin{align*}
\Var(\sigma_\ell )
&\leq \frac{c_d \, \ell^{2d-2}}{\ell^{2d-2+\frac{d-1}{2}}} \sum_{q_1,q_2, q_3, q_4\geq 2} \prod_{i=1}^ 4\frac{|b_{q_i}|}{(q_i-1)!}
\E\Big[\prod_{r=1}^4 H_{q_r-1}(Z)\Big]\\
&=  \frac{c_d }{\ell^{\frac{d-1}{2}}} \E\Big[\Big(\sum_{q \geq 2}\frac{|b_q|}{(q-1)!}H_{q-1}(Z)\Big)^ 4\Big].
\end{align*}
By Assumption \ref{ASSUMPTION},
\begin{align*}
\Var(\sigma_\ell )
&\leq \frac{c_d }{\ell^{\frac{d-1}{2}}}\E[|D \phi(Z)|^4]\stackrel{\ell\to\infty}{\longrightarrow}0.
\end{align*}
This concludes the proof for the case $d\geq 3$. If $d=2$, the above bound still holds but is not enough to prove that $\Var(\sigma_\ell )=O(\ell^ {-1})$, allowing us to reach the optimal bound in Theorem \ref{mainThm}.	Since the proof strategy for the estimation of the integrals for $d=2$ is different and needs a long analysis, we postpone this case to Appendix \ref{d2} (see Proposition \ref{propDIM2}).
\end{proof} 

\section{Uniform boundedness of Malliavin-Sobolev norms}\label{sect:UnifLim}

This section is devoted to the proof of Proposition \ref{UnifLim}, that is, all moments of $|D^{(k)} \X_\ell|_{\mathcal{H}^{\otimes k}}$ and of $|D^{(k)} L\X_\ell|_{\mathcal{H}^{\otimes k}}$ are uniformly bounded in $\ell$.

\medskip

\begin{proof}[Proof of Proposition \ref{UnifLim}]  Without loss of generality, we can assume that $n$ is even. So, we fix $k\in\N$ and $n=2p$, $p\in\N$.  We first prove that $\sup_{\ell \mathrm{\ even}}\E[|D^{(k)} \X_\ell|^{2p}_{\mathcal{H}^{\otimes k}}]$ is finite.

Recall the chaos expansion \eqref{chaos-tildeX} and by using \eqref{DkIntHp}, it easily follows that
$$
D_{y_1,\ldots y_k}^{(k)} \X_\ell= \frac{1}{v_{\ell;d}} \left(\frac{n_{\ell;d}}{\mu_d}\right)^{\frac k2} \sum_{q\geq 2\vee k} \frac{b_q}{(q-k)!}\int_{\mathbb{S}^d} H_{q-k}(T_\ell(x)) \prod_{i=1}^k G_{\ell; d}(\<x,y_i\>) dx. 
$$
Thus,
\begin{align*}
&\E[|D^{(k)} \X_\ell|^{2 p}_{\mathcal{H}^{\otimes k}}]=\E\Big[\Big(\int_{(\mathbb{S}^d)^k} |D^{(k)}_{y_1,\ldots,y_k} \X_\ell|^2 dy_1 dy_2 \ldots dy_k\Big)^{p}\Big]\\
&= \frac{1}{v_{\ell;d}^{2 p}} \E\Big[\Big(\sum_{q_1,q_2 \geq 2 \vee k} \frac{b_{q_1} b_{q_2}}{(q_1-k)!(q_2-k)!} \int_{(\mathbb{S}^d)^2}H_{q_1-k}(T_\ell(y)) H_{q_2-k}(T_\ell(z)) G_{\ell; d}(\<y, z\>)^k dydz \Big)^p  \Big]\\
&=\frac{1}{v_{\ell;d}^{2p}}  \sum_{  q_i \geq 2 \vee k, i=1,\ldots,2p} \prod_{i=1}^{2p} \frac{b_{q_i}}{(q_i-k)!} \prod_{r=1}^{2p} (q_r-k)! \sum_{\{k_{i,j}\}^{2p}_{\underset{i<j}{ i,j=1}} \in A_{q_1-k,\ldots,q_{2p}-k} } \prod_{\underset{i<j}{i,j=1}}^{2p}\frac{1}{k_{i,j}!}\times\\
&\times \int_{(\mathbb{S}^d)^{2p}} \prod_{\underset{i<j}{i,j=1}}^{2p} G_{\ell;d}(\<x_i,x_j\>)^{k_{i,j}} \prod_{s=1}^p G_{\ell; d}(\<x_s, x_{s+p}\>)^k dx,
\end{align*}  
in which we have used \eqref{eqDIAGRAMMA}. We start to estimate the integrals in the above r.h.s.
Using the Cauchy Schwarz inequality and \eqref{connessiSTIMA}, it follows that
\begin{align*}
&\Big|\int_{(\mathbb{S}^d)^{2p}} \prod_{\underset{i<j}{i,j=1}}^{2p} G_{\ell;d}(\<x_i,x_j\>)^{k_{i,j}} \prod_{s=1}^p G_{\ell; d}(\<x_s, x_{s+p}\>)^k dx\Big|\leq \\
&\Big(\int_{(\mathbb{S}^d)^{2p}} \prod_{\underset{i<j}{i,j=1}}^{2p} G_{\ell;d}(\<x_i,x_j\>)^{2k_{i,j}} dx \int_{(\SSd)^{2p}}\prod_{s=1}^p G_{\ell; d}(\<x_s, x_{s+p}\>)^{2k} dx\Big)^{\frac{1}{2}}\leq \frac{C_{d;p}}{\ell^{(d-1)p}}.
\end{align*}
We insert the above estimate and, by using the asymptotics of $v_{\ell;d}$ in \eqref{v-ell} and the representation \eqref{PassDel}, it follows that
\begin{align*}
&\E[|D^{(k)} \X_\ell|^{2 p}_{\mathcal{H}^{\otimes k}}]\leq \\
&\leq\frac{1}{v_{\ell;d}^{2p}}\frac{C_{d;p}}{\ell^{(d-1)p}} \sum_{  q_i \geq 2 \vee k, i=1,\ldots,2p} \left( \prod_{i=1}^{2p} \frac{|b_{q_i}|}{(q_i-k)!}\right) \prod_{r=1}^{2p} (q_r-k)! \sum_{\{k_{i,j}\}^{2p}_{\underset{i<j}{ i,j=1}} \in A_{q_1-k,\ldots,q_{2p}-k} } \prod_{\underset{i<j}{i,j=1}}^{2p}\frac{1}{k_{i,j}!}\\
&\leq \frac{\mathrm{Const} \big(\ell^{(d-1)p}+o(\frac{1}{\ell^{(d-1)p}})\big)}{\ell^{(d-1)p}} \times\sum_{  q_i \geq 2 \vee k, i=1,\ldots,2p} \left( \prod_{i=1}^{2p} \frac{|b_{q_i}|}{(q_i-k)!}\right) \E[\prod_{r=1}^{2p} H_{q_r-k} (Z)]\\
&= \mathrm{Const}(1+o(1))\E\Big[\Big|
\sum_{q \geq 2 \vee k}\frac{|b_q|}{(q-k)!} H_{q-k}(Z)\Big|^{2p}\Big].
\end{align*}
Hereafter $\mathrm{Const}$ denotes a positive constant, possibly changing from a line to another and possibly depending on $d$ and $p$ but  independent of $\ell$. Now, by \eqref{DkTphi} in Assumption \ref{ASSUMPTION}, 
$$
\E\Big[\Big|
\sum_{q \geq 2 \vee k}\frac{|b_q|}{(q-k)!} H_{q-k}(Z)\Big|^{2p}\Big]=\E[|D^ k\phi(Z)|^ {2p}]<\infty,
$$
so that
$$
\sup_{\ell \,\mathrm{even}}\E[|D^{(k)} \X_\ell|^{2 p}_{\mathcal{H}^{\otimes k}}]\leq \mathrm{Const}\, \E[|D^ k\phi(Z)|^ {2p}]<\infty.
$$

Concerning the study of $\E[|D^{(k)} L\X_\ell|^{2p}_{\mathcal{H}^{\otimes k}}]$, by using \eqref{defL} we have
$$
L\X_\ell=-\frac{1}{v_{\ell;d}} \sum_{q\geq 2} \frac{b_q}{(q-1)!} \int_{\SSd} H_q(T_\ell(x))dx.
$$
By comparing this expansion with \eqref{chaos-tildeX}, one deduces that one can repeat the same computations with $b_q$ replaced by $-qb_q$. Hence, 
$$
\sup_{\ell \,\mathrm{even}}\E[|D^{(k)}L \X_\ell|^{2 p}_{\mathcal{H}^{\otimes k}}]\leq \mathrm{Const}\, 
\E\Big[\Big|
\sum_{q \geq 2 \vee k}\frac{q|b_q|}{(q-k)!} H_{q-k}(Z)\Big|^{2p}\Big]
= \mathrm{Const}\, \E[|D^ kL\phi(Z)|^ {2p}]
$$
and this is finite again because of \eqref{DkTphi} in Assumption \ref{ASSUMPTION}. This concludes the proof.

\end{proof}

\appendix

\section{Appendix A}\label{A1}

In this Section we collect some technical proofs. 

\subsection{On Assumption \ref{ASSUMPTION}}\label{appendix1}

\begin{proof}[Proof of Proposition \ref{exponential}]
	It suffices to prove that if there exist $C,R>0$ such that $|a_q|\leq CR^ q$ for every $q\geq 0$ then $\sum_{q\geq 0} \frac{|a_{q}|}{q!} |H_{q}(Z)|$ converges in $L^2(\R,\B(\R),\nu)$ to a r.v. having  all moments. 
	
	We recall that
	$$
	H_q(z)=\sum_{n=0}^ {\lfloor q/2\rfloor} q!\, \frac{(-1)^ n}{2^ nn! (q-2n)!}\, z^ {q-2n},
	$$
	hence
	$$
	|H_q(z)|\leq \sum_{n=0}^ {\lfloor q/2\rfloor} q!\, \frac{1}{2^ nn! (q-2n)!}\, |z|^ {q-2n}.
	$$
	By splitting the cases $q$ even and $q$ odd and by inserting the above estimate, we have
	\begin{align*}
	\sum_{q\geq 0} \frac{|a_{q}|}{q!} |H_{q}(Z)|
	&
	\leq \sum_{q\geq 0} \sum_{n=0}^ {q} \frac{CR^{2q}}{2^ nn! (2q-2n)!}\, |Z|^ {2q-2n}
	+\sum_{q\geq 0} \sum_{n=0}^ {q} \frac{CR^{2q+1}}{2^ nn! (2q+1-2n)!}\, |Z|^ {2q+1-2n}\\
	&
	\leq C\sum_{n\geq 0}\frac{R^ {2n}}{2^ nn!}\sum_{q\geq n}\Big(\frac{(R|Z|)^ {2q-2n}}{ (2q-2n)!}+\frac{(R|Z|)^ {2q+1-2n}}{ (2q+1-2n)!}\Big)\\
	&
	= C\sum_{n\geq 0}\frac{R^ {2n}}{2^ nn!}\sum_{m\geq 0}\frac{(R|Z|)^ {m}}{ m!}
	=C \e^ {R^2+R|Z|}.
	\end{align*} 
	Since the above r.v. has got any moment, the statement follows.

\end{proof}

\subsection{Proof of Lemma \ref{DIAGRAMMA}}\label{Diag} 

This section is devoted to the proof of Lemma \ref{DIAGRAMMA}. First of all, let us recall the diagram formula \cite[Proposition 4.15]{MPbook}. Let $n\geq 1$ and $(Z_1,\ldots,Z_n)$ be a centered Gaussian vector in $\R^ d$. Then for every $q_1,\ldots,q_n\in\N$ one has
\begin{equation}\label{FormDiag}
\E[\prod_{r=1}^n H_{q_r}(Z_r)]=\sum_{G\in \Gamma_{\overline{F}}(q_1,\ldots,q_n)} \prod_{1\leq i<j\leq n} \E(Z_i Z_j)^{k_{ij}(G)}
\end{equation}
where $\Gamma_{\overline{F}}(q_1,\ldots,q_n)$ is the set of no-flat diagram of order $(q_1,\ldots, q_n)$ and $k_{ij}(G)$ is the number of edges from row $i$ to row $j$ of the diagram. Let us recall (see \cite[\S 4.3.1]{MPbook}, in particular the figure at page 97) that a diagram $G$ of order $(q_1,\ldots, q_n)$ is a set of points $\{(i,h): 1\leq i\leq n, 1\leq h\leq q_i\}$ called $\mathit{vertices}$ and a partition of these points into pairs
$$
\{ ((i, h),(j, k)) : 1\leq i\leq j\leq n;\,\,\, 1\leq h \leq q_i,\,\,\, 1\leq k \leq q_j        \}
$$
called $\mathit{edges}$, such that $(i,h)\neq(j,k)$ (self loops are not allowed) and moreover, every vertex of the diagram  is linked to one and only one vertex through an edge. One can graphically represent $G$ by a set of $n$ rows, where the $i$th row contains $q_i$ dots. The $h$th dot (from left to right) of the $i$th row represents the point $(i,h)$. The edges of the diagram are represented as lines connecting the two corresponding dots. A diagram is \textit{no-flat} if for all edges $((i,h),(j,k))$ we have $i\neq j$. It graphically means that we can connect only dots that are in two different rows.

\medskip

\begin{proof}[Proof of Lemma \ref{DIAGRAMMA}]
We start from the diagram formula \eqref{FormDiag}. For a diagram in  $\Gamma_{\overline{F}}(q_1,\ldots,q_n)$, let $R_i$ denote its $i$th row, $i=1,\ldots, n$. Consider the the first row $R_1$. In $R_1$ we have $q_1$ dots; we fix a partition of $q_1$ dots in $n-1$ groups of dots. We order the groups and denote them $R_{1j}$, $j=2,\ldots, n$: $R_{1j}$ is the group of dots in $R_1$ that are linked with dots in the $j$th row. We denote with $k_{1j}$ the number of dots in $R_{1j}$, that coincides with the number of edges connecting row $1$ with row $j$. We fix $k_{12}\in\{0,\ldots, q_1 \}$. There are $\binom{q_1}{k_{12}}$ choices for $k_{12}$ dots in the first row. In general for $j=3,\ldots n$, we fix $k_{1j}=0,\ldots, (q_1-\sum_{h=2}^{j-1} k_{1h}) $ to have that 
$$
\sum_{j=2}^n k_{1j}=q_1.
$$
For $j=3,\ldots,n$ there are $\binom{q_1-\sum_{h=2}^{j-1} k_{1h}}{k_{1j}}$ choices for $k_{1j}$. Then, the number of choices of $\{ k_{1,j} \}_{j=2}^n$ according to the above condition is
$$
\prod_{j=1}^n \binom{q_1-\sum_{r=1}^{j-1} k_{1r}}{k_{1j}}= \frac{q_1!}{\prod_{j=1}^n k_{1j}!}.
$$ 
We recall that $k_{ii}=0$ for $i=1,\ldots, n$ because we are considering no-flat diagrams. In practice we have computed the number of partitions of $q_1$ dots in $n-1$ groups. We can do the same for the other rows. And so we have that the number of partition of $q_i$ that is 
$$
\prod_{j=1}^n \binom{q_i-\sum_{r=1}^{j-1} k_{ir}}{k_{ij}}= \frac{q_i!}{\prod_{j=1}^n k_{ij}!}.
$$ 
Notice that $k_{ij}=k_{ji}$. Now we are able to compute the number of diagrams for fixed $\{k_{ij}\}_{i,j =1}^n$. We recall that $k_{ij}$ represent the number of dots of the $i$th row and of the $j$th row that are linked. There are $k_{ij}!$ way to match the dots.
Then the number of no-flat diagrams for a fixed $\{k_{ij}\}_{i,j =1}^n$ is 
$$
\prod_{i=1}^n \frac{q_i!}{\prod_{j=1}^n k_{ij}!} \prod_{\underset{r<s}{r,s=1}}^n k_{rs}!=\prod_{r=1}^n q_r! \prod_{\underset{i<j}{i,j=1}}^n \frac{1}{k_{ij}}.
$$
In order to conclude, it remains to determine the set of all admissible $\{k_{ij}\}_{i,j =1}^n$. Recalling that, for a fixed no-flat  diagram, $k_{ij}$ is the number of edges connecting row $i$ with row $j$, then of course $k_{ij}=k_{ji}$. Moreover, $k_{ii}=0$ because the diagram is no-flat  and $\sum_{j=1}^n k_{ij}=q_i$ for every $i$, as every vertex belongs to a unique edge. This means that $\{k_{ij}\}_{i,j=1}^ n\in \mathcal{A}_{q_1,\ldots,q_n}$ (see Definition \ref{defA}). The statement now follows.

\end{proof}

\subsection{Proof of Lemma \ref{STIMA}}\label{sect:graph}

Before presenting the proof of Lemma \ref{STIMA}, we start by recalling some elementary concepts of graph theory \cite{Va17}.

A \textit{graph} is a set of point called \textit{nodes} linked together by lines called \textit{edges}. Formally, a graph is a pair $\mathfrak{G}=(V,E)$ of sets, where $V$ is the set of nodes and $E$ is the set of edges. We can identify $E$ with a subset of $V\times V$. Precisely, if $V=\{x_1,\ldots,x_n\}$ and there exists an edge between $x_i$ and $x_j$, then the pair $(x_i,x_j)\in E$.    
A \textit{subgraph} of $\mathfrak{G}=(V,E)$ is a graph $\mathfrak{G}'=(V',E')$ where $V'\subset V$ and $E'$ is the set of all the edges of $E$ that link only nodes in $V'$. We say that a node $x$ has degree $m$ if there are $m$ edges that are incident to $x$, the case $m=0$ meaning that the node is isolated.

A \textit{path} between two nodes $x,y$ of $\mathfrak{G}$ is a sequence of edges connecting $x$ with $y$ and joining a sequence of distinct nodes, so, in particular, all edges of the path are distinct. 
We say that two nodes $x,y$ of a graph $\mathfrak{G}$ are \textit{connected} if $\mathfrak{G}$ contains a path between $x$ and $y$. A graph is said to be \textit{connected} if every pair of nodes in the graph is connected. A \textit{connected component} of a graph $\mathfrak{G}$ is  connected subgraph of the graph that is maximal. We can consider a graph as the union of its connected components.

In our treatment we are interested in a particular class of connected graphs: the trees. A \textit{tree} is a connected graph where each pair of nodes is connected by exactly one path. We first observe that in a tree there exists a non-empty subset of nodes with degree $1$. In fact, equivalently, a tree  is a connected graph in which every subgraph (and in particular the graph itself) contains at least one node with degree $1$. 
Hence when we delete some of $1$ degree nodes, the subgraph that we obtain is also a tree, that has again a subset of new $1$ degree nodes. If we progressively delete the $1$ degree nodes, we finally obtain a empty graph.

The last and most important property (for our treatment) of connected graphs is the following: a connected graph $\mathfrak{G}$ always contains a \textit{spanning tree}, i.e. a subgraph of $\mathfrak{G}$ that is a tree and contains all nodes of $\mathfrak{G}$.
Now we prove Lemma \ref{STIMA}.

\medskip

\begin{proof}[Proof of Lemma \ref{STIMA}] 
Let $\kappa=\{k_{ij}\}_{i,j=1}^{n}\in \mathcal{A}_{q_1-1,\ldots, q_{n}-1}$ (see \eqref{eqDIAGRAMMA}).
%
We extrapolate from $\kappa$ the graph $\mathfrak{G}=(V,E)$ with $ V=\{x_1,\ldots, x_{n}\}$ and $(x_i,x_j)\in E$ iff	$k_{ij}\neq0$. We recall that for every $i=1,\ldots,n$, $k_{ii}=0$, then there are no self loops in $\mathfrak{G}$. 

Let $N$ be the number of the connected components of $\mathfrak{G}$ and we denote with $\mathfrak{G}_h$, $h=1,\ldots, N$ these components. We denote with $m_h$ the number of nodes in $\mathfrak{G}_h$. Then $\sum_{h=1}^N m_h=n$.
We observe that if $x_i$ is a node of $\mathfrak{G}_{h_1}$, $x_j$ is a node of $\mathfrak{G}_{h_2}$ and $h_1\neq h_2$ then $k_{ij}=0$. This justifies the following equality:
\begin{equation}\label{integrale}
\begin{array}{l}
\int_{(\SSd)^{n}} \prod_{\underset{i<j}{i,j= 1}}^{n} G_{\ell;d}(\<x_i,x_j\>)^{2k_{ij}} dx_1\dots dx_{n}\\
=\prod_{h=1}^N \int_{(\SSd)^{m_h}} \prod_{x_{i_r}, x_{i_s} \in \mathfrak{G}_h} G_{\ell ; d}(\<x_{i_r}, x_{i_s})^{2k_{i_r i_s}} dx_{i_1}\ldots dx_{i_{m_h}}.
\end{array}
\end{equation}
Now we observe that if $\mathfrak{G}_h$ is a tree, the $1$ degree nodes of $\mathfrak{G}_h$ are the variables $x_{i_r}$ 
 for which there exists one and only one $i_s$  such that $k_{i_r i_s} \neq 0$. Hence there is one and only one polynomial $G_{\ell;d}$ in the variable $x_{i_r}$ in the integral \eqref{integrale}.
We identify the action of deleting  $1$ degree nodes with that of integrating the polynomial $G_{\ell;d}(\<x_{i_r}, x_{i_s}\>)$ in the variable $x_{i_r}$. 
Our connected components are not always trees, but we know that there always exists the spanning tree. So for all $\mathfrak{G}_h$, we consider the spanning tree $\tilde{ \mathfrak{G}}_h$, and delete the edges of $\mathfrak{G}_h$ that are not in $\tilde{\mathfrak{G}}_h$. This deleting operation corresponds, when studying the integral in the r.h.s. of \eqref{integrale}, with the estimate $|G_{\ell;d}(\<x_{i_r}, x_{i_s}\>)|\leq 1 $ for each pair $(x_{i_r}, x_{i_s})$ giving the deleted edge.
Since in a tree with $m_h$ nodes there are $m_h-1$ edges, the resulting estimate consists in integrating $m_h-1$ polynomials.

It follows that
\begin{eqnarray*}
\int_{(\SSd)^{n}} \prod_{\underset{i<j}{i,j= 1}}^{n} G_{\ell;d}(\<x_i,x_j\>)^{2k_{ij}} dx
&=&\prod_{h=1}^N \int_{(\SSd)^{m_h}} \prod_{x_{i_r}, x_{i_s} \in \mathfrak{G}_h} G_{\ell ; d}(\<x_{i_r}, x_{i_s})^{2k_{i_r i_s}} dx_{i_1}\ldots dx_{i_{m_h}}\nonumber\cr
&\leq&
\prod_{h=1}^N \int_{(\SSd)^{m_h}} \prod_{x_{i_r}, x_{i_s} \in \tilde{ \mathfrak{G}}_h} G_{\ell ; d}(\<x_{i_r}, x_{i_s})^{2k_{i_r i_s}} dx_{i_1}\ldots dx_{i_{m_h}}\nonumber\cr
&\leq&
\prod_{h=1}^N \int_{(\SSd)^{m_h}} \prod_{x_{i_r}, x_{i_s} \in \tilde {\mathfrak{G}}_h} G_{\ell ; d}(\<x_{i_r}, x_{i_s})^{2} dx_{i_1}\ldots dx_{i_{m_h}}\nonumber\cr
&
\leq&
\prod_{h=1}^N  \frac{(8 \mu_d \mu_{d-1} c_{2;d} )^{m_h}\mu_d^{N}}{\ell^{(d-1)(m_h-1)}}
=\frac{(8 \mu_d \mu_{d-1} c_{2;d})^{n-N} \mu_d^{N}}{\ell^{(d-1)(\sum_{h=1}^Nm_h-N)}}\nonumber\cr
&=&\frac{(8 \mu_d \mu_{d-1} c_{2;d})^{n-N} \mu_d^{N}}{\ell^{(d-1)(n-N)}}.
\end{eqnarray*}
We end by observing that the maximum number $N$ of connected components in a graph that contains $n=2p$ nodes is $p$, when there aren't $0$ degree nodes. Moreover there are exactly $p$ connected components when all subgraph contains exactly $2$ nodes. Then, being $8 \mu_d \mu_{d-1} c_{2;d}>1$, we have
\begin{equation}
\int_{(\SSd)^{2p}} \prod_{\underset{i<j}{i,j= 1}}^{2p} G_{\ell;d}(\<x_i,x_j\>)^{2k_{ij}} dx\leq \frac{C_{d;p}}{\ell^{(d-1)p}}
\end{equation}
where $C_{d;p}=(2 (d-1)!\mu_d^2)^{2p} \mu_d^p$, thus concluding the proof.

\end{proof}

\section{Appendix B}\label{d2} 

We recall that in Section \ref{3.2} we proved that $\Var(\sigma_\ell)=O(\ell^{-(d-1)/2})$ for $d\geq 2$. When $d=2$, it gives $\Var(\sigma_\ell)=O(\ell^{-1/2})$, which is not enough fine for our purposes. This section is then devoted to the following improved result:
	\begin{proposition}\label{propDIM2}
		Let $\sigma_\ell$ be the Malliavin covariance of $\X_\ell$. Under Assumption \ref{ASSUMPTION}, for $d=2$ we have
		\begin{equation}\label{eqVarDIM2}
		\Var(\sigma_\ell )= O\left ( \ell^{-1}\right ). 
		\end{equation}
	\end{proposition} 
	The improved estimate \eqref{eqVarDIM2} will be reached by using a completely different technique, based on appropriate estimates for the Legendre polynomials involving the analysis of Gaunt integrals. Once all the details will be tuned, the proof of Proposition \ref{propDIM2} will turn out to be an easy consequence (see Section \ref{B3}).

\subsection{Convolutions of Gaunts integrals}

\begin{definition}
For $d\geq 2$, $q\in \N$ and $n_1,\ldots, n_q \in \{0, \ldots, n_{\ell;d}\}$, the generalized Gaunt integral on $\SSd$ is:
	$$
	\G_{\ell n_1, \ldots, \ell n_q}=\int_{\SSd} \prod_{i=1}^q Y_{\ell n_i}(x) dx.
	$$
\end{definition} 
The following lemma generalizes Lemma 1.5 in \cite{Ros19}. 
\begin{lemma}\label{GauntProd}
For $q\in \N$ and $n, n_1, \ldots, n_q \in \{0,\ldots, n_{\ell;d}\}$ one has
	$$
	\sum_{m_1, \ldots, m_r}^{n_{\ell;d}} \G_{\ell m_1, \ldots, \ell m_r, \ell n}\, \G_{\ell m_1, \ldots, \ell m_r, \ell n_1, \ldots,\ell n_q}= \Big(\frac{n_{\ell;d}}{\mu_d} \Big)^{r-1}\hat \gamma_{\ell ;r} \,\,\G_{\ell n, \ell n_1,\ldots, \ell n_q}
	$$  
where
\begin{equation}\label{gammal}
\hat\gamma_{\ell; r}=n_{\ell;d}\frac{\mu_{d-1}}{\mu_d}\int_{-1}^1 G_{\ell;d}(t)^{r+1} \big(\sqrt{1-t^2}\big)^{d-2} dt.
\end{equation}
\end{lemma}

\begin{proof}
	From \eqref{covT} we have
	\begin{align*}
	&\sum_{m_1, \ldots, m_r} \G_{\ell m_1, \ldots, \ell m_r, \ell n}\, \G_{\ell m_1, \ldots, \ell m_r, \ell n_1, \ldots, n_q}=\\
	&=\sum_{ m_1, \ldots,  m_r} \int_{(\SSd)^2} \prod_{i=1}^r Y_{\ell m_i}(x)Y_{\ell m_i}(y)  Y_{\ell n}(x) \prod_{j=1}^q Y_{\ell n_j}(y) dxdy\\
	&=\Big(\frac{n_{\ell;d}}{\mu_d}\Big)^r \int_{(\SSd)^2} G_{\ell;d}(\<x,y\>)^r \,\,Y_{\ell n}(x) \prod_{j=1}^q Y_{\ell n_j}(y) dxdy.
	\end{align*}
	Since $\Big( \Big( \frac{\mu_{d-1} n_{\ell;d}}{\mu_d}\Big)^\frac{1}{2} G_{j;d}  \Big)_{j=0}^{r \ell}$ is an orthonormal system on $[-1,1]$ with the weight function $(1-t^2)^{d/2-1}$ (see \cite{Sze39}), we can write
	\begin{equation}\label{sviluppo}
	G_{\ell;d}(t)^r =\sum_{j=0}^{r \ell} \gamma_{j,\ell; r} G_{j;d}(t)
	\end{equation}
where, for $j=0,1,\ldots,r\ell$
\begin{equation}\label{gammalr}
\gamma_{j,\ell; r}=n_{j;d}\frac{\mu_{d-1}}{\mu_d}\int_{-1}^1 G_{\ell;d}(t)^r G_{j;d}(t)\big(\sqrt{1-t^2}\big)^{d-2} dt.
\end{equation}
	Substituting \eqref{sviluppo}, we have
	\begin{align*}
	&\sum_{m_1, \ldots, m_r} \G_{\ell m_1, \ldots, \ell m_r, \ell n}\, \G_{\ell m_1, \ldots, \ell m_r, \ell n_1, \ldots, n_q}\\
	&=\Big(\frac{n_{\ell;d}}{\mu_d}\Big)^r \int_{(\SSd)^2} \sum_{j=0}^{r\ell} \gamma_{j, \ell; r} G_{j;d}(\<x,y\>)\,\,Y_{\ell n}(x) 
	\prod_{i=1}^q Y_{\ell n_i}(y)dxdy\\
	&=\Big(\frac{n_{\ell;d}}{\mu_d}\Big)^{r-1}\sum_{j=0}^{r\ell}  \gamma_{j,\ell; r} \int_{(\SSd)^2} \sum_{h=0}^{n_{\ell;d}} Y_{j h}(x) Y_{j h}(y)\,\,Y_{\ell n}(x) 
\prod_{i=1}^q Y_{\ell n_i}(y)dxdy\\
	&=\Big(\frac{n_{\ell;d}}{\mu_d}\Big)^{r-1}\sum_{j=0}^{r\ell}  \gamma_{j,\ell; r}  \sum_{h=0}^{n_{\ell;d}} \underbrace {\int_{\SSd} Y_{j h}(x) Y_{\ell n}(x) dx }_{=\1_{j=\ell}\1_{h=n}}\int_{\SSd}Y_{j h}(y) 
	\prod_{i=1}^q Y_{\ell n_i}(y)dy\\
	&=\Big(\frac{n_{\ell;d}}{\mu_d}\Big)^{r-1}   \gamma_{\ell,\ell ; r}   \, \int_{\SSd}Y_{\ell n}(y) 
	\prod_{i=1}^q Y_{\ell n_i}(y)dy\\
	&=\Big(\frac{n_{\ell;d}}{\mu_d}\Big)^{r-1}   \gamma_{\ell,\ell ; r}  \,\G_{\ell n, \ell n_1, \ldots, \ell n_q}.
	\end{align*}
Since $\hat{\gamma}_{\ell;r}= \gamma_{\ell,\ell ; r} $, the statement follows.
\end{proof}

Now we take advantage of Lemma \ref{GauntProd} to prove the following results. 
\begin{lemma}\label{A_1}
	For $p, q \geq 2$ there exists a positive constant $c_d$ such that
	\begin{align*}
	\int_{(\SSd)^3} G_{\ell;d}(\<x_1,x_2\>) &G_{\ell;d}(\<x_1,x_4\>)^p G_{\ell;d}(\<x_2,x_4\>)^ q dx_1dx_2dx_4 \\&= c_d \int_{(\SSd)^2}G_{\ell;d} (\<x,y\>)^{q+1} dxdy \int_{(\SSd)^2 }G_{\ell;d} (\<x,y\>)^{p+1} dxdy
	\end{align*}
\end{lemma}
\begin{proof}
	From Lemma \ref{GauntProd} we have
	\begin{align*}
	&\int_{(\SSd)^3} G_{\ell;d}(\<x_1,x_2\>) G_{\ell;d}(\<x_1,x_4\>)^p G_{\ell;d}(\<x_2,x_4\>)^ q dx_1dx_2dx_4\\
	&=\Big(\frac{\mu_d}{n_{\ell;d}}\Big)^{p+q+1} \sum_{r, n_1,\ldots, n_p} \sum_{m_1,\ldots, m_q} \G_{\ell r , \ell n_1, \ldots, \ell n_p} \G_{\ell r, \ell m_1,\ldots, \ell m_q} \G_{\ell n_1,\ldots, n_p,\ell m_1, \ldots, \ell m_q}\\
	&=\Big(\frac{\mu_d}{n_{\ell;d}}\Big)^{q+2} \hat \gamma_{\ell; p}   \sum_{r, m_1,\ldots, m_q} \G_{\ell r , \ell m_1, \ldots, \ell m_q}^2 \\
	&=c_d \int_{(\SSd)^2}G_{\ell;d} (\<x,y\>)^{q+1} dxdy \int_{(\SSd)^2 }G_{\ell;d} (\<x,y\>)^{p+1} dxdy.
	\end{align*}
\end{proof}
\begin{lemma}\label{B}
There exists a positive constant $ c_d>0$ such that for every $q_1, q_2\geq 2$, $q_3\geq 1$ 
	\begin{align*}
	&\int_{(\SSd)^4} G_{\ell;d}(\<x_1, x_2\>)G_{\ell;d}(\<x_1, x_4\>)^{q_1} G_{\ell;d}(\<x_2, x_3\>)^{q_2}G_{\ell;d}(\<x_3, x_4\>)^{q_3}dx\\
	&=c_d \int_{(\SSd)^2} G_{\ell; d}(\<x,y\>)^{q_1+1}dxdy\int_{(\SSd)^2} G_{\ell; d}(\<x,y\>)^{q_2+1}dxdy\int_{(\SSd)^2} G_{\ell; d}(\<x,y\>)^{q_3+1}dxdy
	\end{align*} 
\end{lemma}
\begin{proof}
	From Lemma \ref{GauntProd} we have
	\begin{align*}
	&\int_{(\SSd)^4} G_{\ell;d}(\<x_1, x_2\>)G_{\ell;d}(\<x_1, x_4\>)^{q_1} G_{\ell;d}(\<x_2, x_3\>)^{q_2}G_{\ell;d}(\<x_3, x_4\>)^{q_3}dx\\
	&=\Big(\frac{\mu_d}{n_{\ell;d}}\Big)^{q_1+q_2+q_3+1}\sum_{\underset{r, n_1,\ldots n_{q_1}}{m_1,\ldots, m_{q_2},t_1,\ldots, t_{q_3}}} \G_{\ell r, \ell n_1, \ldots, \ell n_{q_1}}\G_{\ell r, \ell m_1, \ldots, \ell m_{q_2}}\G_{ \ell n_1, \ldots, \ell n_{q_1}, \ell t_1,\ldots, \ell t_{q_3}} \G_{ \ell m_1, \ldots, \ell m_{q_2},\ell t_1, \ldots, \ell t_{q_3}}\\
	&=\Big(\frac{\mu_d}{n_{\ell;d}}\Big)^{q_2+q_3+2}\hat \gamma_{\ell;q_1} \sum_{r,m_1,\ldots, m_{q_2},t_1,\ldots, t_{q_3}}\G_{\ell r, \ell m_1, \ldots, \ell m_{q_2}}\G_{ \ell m_1, \ldots, \ell m_{q_2},\ell t_1, \ldots, \ell t_{q_3}}\G_{\ell r, \ell t_1,\ldots, \ell t_{q_3}}\\
	&=\Big(\frac{\mu_d}{n_{\ell;d}}\Big)^{q_3+3} \hat\gamma_{\ell; q_1} \hat \gamma_{\ell; q_2} \sum_{r,t_1,\ldots, t_{q_3}}\G_{\ell r, \ell t_1,\ldots, \ell t_{q_3}}^2\\
	&=c_d \int_{(\SSd)^2} G_{\ell; d}(\<x,y\>)^{q_1+1}dxdy\int_{(\SSd)^2} G_{\ell; d}(\<x,y\>)^{q_2+1}dxdy\int_{(\SSd)^2} G_{\ell; d}(\<x,y\>)^{q_3+1}dxdy.
	\end{align*}
\end{proof}
\begin{lemma}\label{C}
There exists a positive constant $c_d$ such that for $p_1, p_2\geq 2$ and $p_3\geq 0$, we have
	\begin{align*}
	&\int_{(\SSd)^4} G_{\ell;d}(\<x_1,x_2\>)G_{\ell;d}(\<x_1,x_4\>)^{p_1} G_{\ell;d}(\<x_2, x_3\>)^{p_2} G_{\ell;d}(\<x_2,x_4\>)^{p_3}
	G_{\ell;d}(\<x_3,x_4\>)dx=\\
	&= c_d \int_{(\SSd)^2} G_{\ell;d}(\<x,y\>)^{p_1 +1} dxdy.  \int_{(\SSd)^2} G_{\ell;d}(\<x,y\>)^{p_2+1} dxdy.  \int_{(\SSd)^2} G_{\ell;d}(\<x,y\>)^{p_3+2} dxdy. 
	\end{align*}

%
%
\end{lemma}
\begin{proof}
	We have that
	\begin{align*}
	& \int_{(\SSd)^4} G_{\ell;d}(\<x_1,x_2\>)G_{\ell;d}(\<x_3,x_4\>)G_{\ell;d}(\<x_1,x_4\>)^{p_1} G_{\ell;d}(\<x_2,x_3\>)^{p_2} G_{\ell;d}(\<x_2,x_4\>)^{p_3}dx\\
	&=\Big(\frac{\mu_d}{n_{\ell;d}}\Big)^{p_1+p_2+p_3+2}\sum_{r_1,r_2}\sum_{m_1,\ldots, m_{p_1}}\sum_{n_1,\ldots, n_{p_2}} \times \\
	&\times \sum_{s_1,\ldots, s_{p_3}}\G_{\ell r_1, \ell m_1,\ldots, \ell m_{p_1}} \G_{\ell r_1, \ell n_1, \ldots, \ell n_{p_2}, \ell s_1, \ldots, \ell s_{p_3}}\G_{\ell r_2, \ell n_1, \ldots, \ell n_{p_2}}\G_{\ell r_2,  \ell m_1 ,\ldots, \ell m_{p_1},\ell s_1, \ldots, \ell s_{p_3}}.
	\end{align*}
	From Lemma \ref{GauntProd} we have
	\begin{align*}
	& \sum_{m_1,\ldots, m_{p_1}}  \G_{\ell r_1, \ell m_1, \ldots, \ell m_{p_1}}\G_{\ell r_2,  \ell m_1 ,\ldots, \ell m_{p_1},\ell s_1, \ldots, \ell s_{p_3}}=\Big(\frac{n_{\ell;d}}{\mu_d}\Big)^{p_1-1} \gamma_{\ell; p_1 } \G_{\ell r_1,\ell r_2, \ell s_1,\ldots, \ell s_{p_3}}
	\end{align*}
	and 
	\begin{align*}
	\sum_{n_1,\ldots,  n_{p_2}}  \G_{\ell r_1, \ell n_1, \ldots, \ell n_{p_2}, \ell s_1, \ldots, \ell s_{p_3}}\G_{\ell r_2, \ell n_1, \ldots, \ell n_{p_2}}  =\Big(\frac{n_{\ell;d}}{\mu_d} \Big)^{p_2-1}\hat\gamma_{\ell;p_2}\,   \G_{\ell r_1,\ell r_2, \ell s_1,\ldots, \ell s_{p_3}}.
	\end{align*}
	Then, by inserting,
	\begin{align*}
	& \int_{(\SSd)^4} G_{\ell;d}(\<x_1,x_2\>)G_{\ell;d}(\<x_3,x_4\>)G_{\ell;d}(\<x_1,x_4\>)^{p_1} G_{\ell;d}(\<x_2,x_3\>)^{p_2} G_{\ell;d}(\<x_2,x_4\>)^{p_3}dx\\
	&=\Big(\frac{\mu_d}{n_{\ell;d}}\Big)^{p_3+4}\hat\gamma_{\ell; p_1}\hat  \gamma_{\ell;p_2}\,\sum_{r_1, r_2}\sum_{s_1,\ldots, s_{p_3}}\G_{\ell r_1, \ell r_2, \ell s_1,\ldots, \ell s_{p_3}}^2\\
	&=\Big(\frac{\mu_d}{n_{\ell;d}}\Big)^2 \hat\gamma_{\ell; p_1}  \hat\gamma_{\ell;p_2} \int_{(\SSd)^2} G_{\ell;d}(\<x,y\>)^{p_3+2} dxdy\\
	&=c_d \int_{(\SSd)^2} G_{\ell;d}(\<x,y\>)^{p_1 +1} dxdy.  \int_{(\SSd)^2} G_{\ell;d}(\<x,y\>)^{p_2+1} dxdy.  \int_{(\SSd)^2} G_{\ell;d}(\<x,y\>)^{p_3+2} dxdy.
	\end{align*}
	
\end{proof}

\subsection{Integrals of Legendre polynomials} 

We consider now the case $d=2$, in which the Gegenbauer polynomials coincide with the Legendre polynomials: $G_{\ell; 2}\equiv P_\ell$, $\ell\geq 0$.

For $q_1,\ldots, q_4 \in \N$ and $\kappa=\{k_{ij}\}_{i,j=1}^ 4\in \mathcal{A}_{q_1-1, \ldots, q_4-1}$ (see Definition \ref{defA}), we define
\begin{equation}\label{I}
\mathfrak{I}_{q_1,\ldots, q_4, \kappa}(\ell)=\int_{(\SS2)^4} P_{\ell}(\<x_1,x_2\>)^{k_{12}+1}\prod_{i<j, i<3}P_{\ell}(\<x_i,x_j\>)^{k_{ij}} \, P_{\ell}(\<x_3,x_4\>)^{k_{34}+1}dx.
\end{equation}
We recall the set  $\mathcal{C}_{q_1-1,\ldots,q_4-1}$ introduced in the proof of Proposition \ref{Conv0}: $\mathcal{C}_{q_1-1,\ldots,q_4-1}=\mathcal{A}_{q_1-1,\ldots,q_4-1}\setminus \mathcal{N}_{q_1-1,\ldots,q_4-1}$, where $\kappa= \{k_{ij}\}_{i,j=1}^ 4\mathcal{N}_{q_1-1,\ldots,q_4-1}$  if and only if
$k_{12}=q_1-1=q_2-1$, $k_{34}=q_3-1=q_4-1$ and $k_{ij}=0$ if $(i,j)\neq (1,2), (3,4)$. We recall that the variance of the Malliavin covariance has a series expansion in which the terms associated with $\mathcal{N}_{q_1-1,\ldots,q_4-1}$ are null, following a simplification between the expansion of the mean of the square and the square of the mean.

The goal of the present section is to prove the following result:
\begin{proposition} \label{prop-I}
There exists $C>0$ such that for every $q_1,\ldots, q_4 \geq 2$ and $\kappa=\{k_{ij}\}_{i,j=1}^ 4\in \mathcal{C}_{q_1-1, \ldots, q_4-1}$ one has
\begin{equation}\label{estd2}
|\mathfrak{I}_{q_1,\ldots, q_4, \kappa}(\ell)|\leq \frac{C}{\ell^3}.
\end{equation}
\end{proposition}
The proof of (\ref{estd2}) is quite long and heavy, so we split our analysis in five different lemmas built according to the number of the $k_{ij}$'s in $\kappa$ that are not null, taking advantage of some symmetry properties. As an immediate consequence, the proof Proposition \ref{prop-I} will follow from the auxiliary lemmas \ref{lemmaI-1}--\ref{lemmaI-5} we are going to introduce.

We set
\begin{equation}\label{R}
R_{q_1,\ldots,q_4,\kappa}=\#\{k_{ij}: k_{ij}\neq 0, i<j \}
\end{equation}
and we divide the cases in accordance with $R_{q_1,\ldots,q_4,\kappa}$. Clearly, $R_{q_1,\ldots,q_4,\kappa}\neq 0$ and $R_{q_1,\ldots,q_4,\kappa}\neq 1$, the latter because the 4 equations $\sum_{l=1}^4 k_{il}=q_i-1$, $i=1,\ldots,4$, must be satisfied. So $2\leq \R_{q_1,\ldots,q_4,\kappa}\leq 6$. Notice that if $\kappa\in \mathcal{N}_{q_1-1,\ldots,q_4-1}$ then $R_{q_1,\ldots,q_4,\kappa}=2$.

In the following lemmas, we take $q_1,\ldots,q_4\geq 2$, $\kappa\in \mathcal{A}_{q_1-1,\ldots,q_4-1}$ and we  consider $\mathfrak{I}_{q_1,\ldots, q_4, \kappa}(\ell)$ and $R_{q_1,\ldots,q_4,\kappa}$ defined in \eqref{I} and \eqref{R} respectively.

\begin{lemma}\label{lemmaI-1}
If $R_{q_1,\ldots,q_4,\kappa}=2$ and $\kappa\notin \mathcal{N}_{q_1-1,\ldots,q_4-1}$,
one has
\begin{equation}\label{asymp-R=2}
|\mathfrak{I}_{q_1,\ldots, q_4, \kappa}(\ell)|\leq \frac{C}{\ell^ {3}},
\end{equation}
$C>0$ denoting a constant independent of $q_1,\ldots,q_4$ and of $\kappa$.

\end{lemma}

\begin{proof}
If $\kappa\notin \mathcal{N}_{q_1-1,\ldots,q_4-1}$ then either $k_{23},k_{14}\neq 0$ or $k_{13},k_{24}\neq 0$. The estimates in these cases are similar, so we only deal with the former one.

According to  $\kappa=\{k_{ij}\}_{i,j=1}^ 4\in \mathcal{A}_{q_1-1, \ldots, q_4-1}$ in Definition \ref{defA}, if $k_{23}, k_{14} \neq 0$ then $k_{14}=q_1-1=q_4-1$ and $k_{23}=q_2-1=q_3-1$. This implies $q_1=q_4$ and $q_2=q_3$ and we prove that
		\begin{equation*}
		|\mathfrak{I}_{q_1,\ldots, q_4, \kappa}(\ell)|\leq  \frac{C}{\ell^{5}}\times \ell^ {\#\{i\in\{1,2\}\,:\,q_i=2\}}\times (\log \ell) ^ {\#\{i\in\{1,2\}\,:\,q_i=4\}}.
		\end{equation*}
		Hereafter, throughout this section, $C$ denotes a positive constant that may vary from a line to another but is independent of $q_1,\ldots,q_4$ and $\kappa$.
		In this case, $\mathfrak{I}_{q_1,\ldots, q_4, \kappa}(\ell)$ becomes
		\begin{align*}
		\mathfrak{I}_{q_1,\ldots, q_4, \kappa}(\ell)= &\int_{(\SS2)^4} P_\ell(\<x_1,x_2\>) P_\ell(\<x_1,x_4\>)^{q_1-1} P_\ell(\<x_2,x_3\>)^{q_2-1}P_\ell(\<x_3,x_4\>)dx.\\
		\end{align*}
        We start to consider the case $q_1=2$. Using twice the reproducing formula \eqref{G1}, we have
        \begin{align*}
        |\mathfrak{I}_{q_1,\ldots, q_4, \kappa}(\ell)|&=\Big|\int_{(\SS2)^4} P_\ell(\<x_1,x_2\>) P_\ell(\<x_1,x_4\>) P_\ell(\<x_2,x_3\>)^{q_2-1} P_\ell(\<x_3,x_4\>)dx\Big|\\
        &=\Big(\frac{4\pi}{2\ell+1}\Big)^2 \Big|\int_{(\SS2)^4}  P_\ell(\<x_2,x_3\>)^{q_2} dx_2 dx_3\Big|\\
        &\leq \frac{C}{(2\ell+1)^2}\Big( \1_{q_2=2}\frac{1}{2\ell+1}+\1_{q_2=4}\frac{\log \ell}{\ell^2} +\1_{q_2=3, q_2\geq 5} \frac{1}{\ell^2}     \Big)\\
        &\leq \frac C{\ell^3}.
        \end{align*}
		If $q_3=2$ one proceeds similarly. Finally consider the case $q_1, q_2 \geq 3$. From Lemma \ref{B} and Proposition \ref{GegProp}, 		we have
        \begin{align*}
        |\mathfrak{I}_{q_1,\ldots, q_4, \kappa}(\ell)|&=\Big|\int_{(\SS2)^4} P_\ell(\<x_1,x_2\>) P_\ell(\<x_1,x_4\>)^{q_1-1} P_\ell(\<x_2,x_3\>)^{q_2-1} P_\ell(\<x_3,x_4\>)dx\Big|\\
        &= C \int_{(\SS2)^4}P_\ell(\<x_1,x_4\>)^{q_1} dx \int_{(\SS2)^4}  P_\ell(\<x_2,x_3\>)^{q_2}  dx \int_{\SS2} P_\ell(\<x_3, x_4\>)^2 dx\\ 
        &\leq \frac{C}{(2\ell+1)^5}(\log \ell)^{\#\{i=1,2:\, q_i=4\}}.
        \end{align*}		
		The statement now holds.
\end{proof}

\begin{lemma}\label{lemmaI-2}
If $R_{q_1,\ldots,q_4,\kappa}=3$, one has
 \begin{equation}\label{asymp-R=3}
|\mathfrak{I}_{q_1,\ldots, q_4, \kappa}(\ell)|\leq \frac{C}{\ell^ {3}},
\end{equation}
$C>0$ denoting a constant independent of $q_1,\ldots,q_4$ and of $\kappa$.
\end{lemma}

\begin{proof}
$R_{q_1,\ldots,q_4,\kappa}=3$ gives 20 different cases, but thanks to some symmetry properties, we can group them in 5 main cases.

First, let us observe that there are instances that are not in accordance with the condition $\kappa= \{k_{ij}\}_{i,j=1}^ 4\in \mathcal{A}_{q_1-1, \ldots, q_4-1}$:
\begin{multicols}{2}
	\begin{enumerate}[label=(\alph*)]
			\item $k_{23},k_{24},k_{34}\neq 0$;
			\item $ k_{13},k_{14},k_{34}\neq 0$; 
			\item $k_{12}, k_{14}, k_{24}\neq 0$;  
			\item $k_{12},k_{13},k_{23}\neq 0$.
		\end{enumerate}
\end{multicols}
So, we proceed with the cases that can be really verified.

\smallskip

\noindent
\textbf{Case 1:} we assume here that there exists $i=1,2,3,4$ where $k_{ij}\neq 0$ for every $j\neq i$, 
 that is,
\begin{multicols}{2}
		\begin{enumerate}[label=(\alph*)]
	\item $k_{12}, k_{13}, k_{14}\neq 0$;
	\item $k_{12}, k_{23}, k_{24}\neq 0$; 
	\item $k_{13}, k_{23}, k_{34}\neq 0$;
	\item $k_{14}, k_{24}, k_{34}\neq 0$.
\end{enumerate}		 
\end{multicols}

Figure \ref{R=3-case1} shows all non flat diagrams with exactly 3 non null edges satisfying the above conditions.

\begin{figure}[h]
	{\footnotesize 	\begin{center}
		\begin{tikzpicture}
		\node[shape=circle,draw](A) at (1,3) {$1$};
		\node[shape=circle,draw](B) at (3,3) {$2$};
		\node[shape=circle,draw](C) at (3,1) {$3$};
		\node[shape=circle,draw](D) at (1,1) {$4$};
		\draw (A) -- (B);
		\draw (A) -- (D);
		\draw (A) -- (C);
		\draw (2,0) node{(a)};
		\node[shape=circle,draw](A) at (4.5,3) {$1$};
		\node[shape=circle,draw](B) at (6.5,3) {$2$};
		\node[shape=circle,draw](C) at (6.5,1) {$3$};
		\node[shape=circle,draw](D) at (4.5,1) {$4$};
		\draw (B) -- (A);
		\draw (B) -- (C);
		\draw (B) -- (D);
		\draw (5.5,0) node{(b)};
		\node[shape=circle,draw](A) at (8,3) {$1$};
		\node[shape=circle,draw](B) at (10,3) {$2$};
		\node[shape=circle,draw](C) at (10,1) {$3$};
		\node[shape=circle,draw](D) at (8,1) {$4$};
		\draw (C) -- (A);
		\draw (C) -- (B);
		\draw (C) -- (D);
		\draw (9,0) node{(c)};
		\node[shape=circle,draw](A) at (11.5,3) {$1$};
		\node[shape=circle,draw](B) at (13.5,3) {$2$};
		\node[shape=circle,draw](C) at (13.5,1) {$3$};
		\node[shape=circle,draw](D) at (11.5,1) {$4$};
		\draw (D) -- (A);
		\draw (D) -- (B);
		\draw (D) -- (C);
		\draw (12.5,0) node{(d)};
		\end{tikzpicture}
	\end{center}
}
	\caption{\small $R_{q_1,\ldots,q_4,\kappa}=3$, Case 1}\label{R=3-case1}
\end{figure}
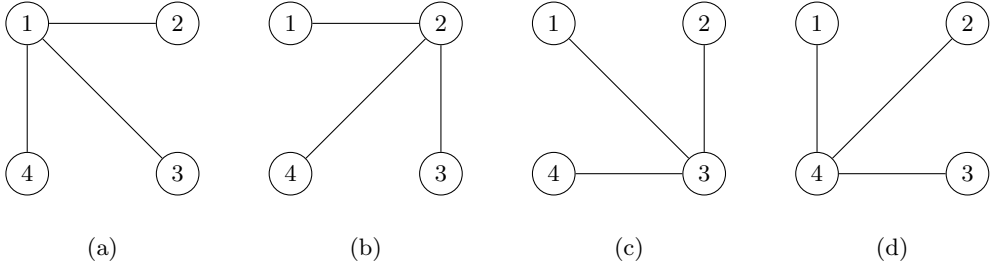

		Instances (a)--(d) can all be studied in a similar way thanks to symmetry properties, so we consider the latter, that is,  $k_{14}, k_{24}, k_{34}\neq 0$, and we prove that
		\begin{equation}\label{asymp2}
			|\mathfrak{I}_{q_1,\ldots, q_4, \kappa}(\ell)|\leq \frac{C}{\ell^{6}}\times\frac{\ell^{\#\{i=1,2,3\,:\,q_i =2\}}}{\ell^{\1_{q_3=2, q_1\neq 2,q_2\neq 2}} }\times (\log \ell)^{\#\{1=1,2\,:\, q_i=4\} +\1_{q_3=3}}          .
		\end{equation}

		The constraints on the $k_{ij}$'s imply that $k_{14}=q_1-1$, $k_{24}= q_2-1$ and $k_{34}=q_3-1$, hence
		$$
		\mathfrak{I}_{q_1,\ldots, q_4, \kappa}(\ell)=\int_{(\SS2)^4} P_\ell(\<x_1, x_2\>)P_\ell(\<x_1, x_4\>)^{q_1-1}P_\ell(\<x_2, x_4\>)^{q_2-1}P_\ell(\<x_3, x_4\>)^{q_3}dx.
		$$
		Applying Lemma \ref{B} and Proposition \ref{GegProp}, for $q_1,q_2\geq 3$ and $q_3\geq 2$ we obtain, 
		\begin{align*}
		|\mathfrak{I}_{q_1,\ldots, q_4, \kappa}(\ell)|&\leq C\, \big|\int_{(\SS2)^2} P_\ell(\<x,y\>)^{q_1} dxdy\, \int_{(\SS2)^2}P_\ell(\<x,y\>)^{q_2} dxdy\,\int_{(\SS2)^2} P_\ell(\<x,y\>)^{q_3+1} dxdy\big| \\
		& \leq  \frac{C}{\ell^{6}}(\log \ell)^{\#\{i=1,2 : q_i=4\}+\1_{q_3=3}}
		\end{align*} 
		and \eqref{asymp2} holds.		
		If $q_1=q_2=q_3=2$ we have
		\begin{align*}
		|\mathfrak{I}_{q_1,\ldots, q_4, \kappa}(\ell)|&=\big|\int_{(\SS2)^3}P_\ell(\<x_2, x_4\>)P_\ell(\<x_3, x_4\>)^2dx_2 dx_3 dx_4 \int_{\SS2}P_\ell(\<x_1, x_2\>)P_\ell(\<x_1, x_4\>)dx_1 \big|\\
		&=\frac{C}{(2\ell+1)^3}
		\end{align*}
		and \eqref{asymp2} follows.
		If $q_1=q_3=2$ and $q_2\geq 3$ (or $q_2=q_3=2$ and $q_1\geq 3$)	we have	
		\begin{align*}
		|\mathfrak{I}_{q_1,\ldots, q_4, \kappa}(\ell)|&=\big|\int_{(\SS2)^3}P_\ell(\<x_2, x_4\>)^{q_2-1}P_\ell(\<x_3, x_4\>)^2 dx_2 dx_3 dx_4 \int_{\SS2}P_\ell(\<x_1, x_2\>)P_\ell(\<x_1, x_4\>)dx_1 \big|\\
		&=\frac{C}{2\ell+1}\big|\int_{(\SS2)^3}P_\ell(\<x_2, x_4\>)^{q_2}P_\ell(\<x_3, x_4\>)^2 dx_2 dx_3 dx_4  \big|\\
		&\leq \frac{C}{\ell^{4}}(\log \ell)^{\1_{q_2=4}} 
		\end{align*}		
		and \eqref{asymp2} again holds.	
		If $q_1=2$ and $q_2,q_3\geq 3$ (or $q_2=2$ and $q_1,q_3\geq 3$)		
		\begin{align*}
		&|\mathfrak{I}_{q_1,\ldots, q_4, \kappa}(\ell)|=\big|\int_{(\SS2)^3}P_\ell(\<x_2, x_4\>)^{q_2-1}P_\ell(\<x_3, x_4\>)^{q_3}dx_2 dx_3 dx_4 \int_{\SS2}P_\ell(\<x_1, x_2\>)P_\ell(\<x_1, x_4\>)dx_1 \big|\\
		&=\frac{C}{2\ell+1}\big|\int_{(\SS2)^3}P_\ell(\<x_2, x_4\>)^{q_2}P_\ell(\<x_3, x_4\>)^{q_3} dx_2 dx_3 dx_4  \big|\\
		&\leq \frac{C}{\ell^{5}} (\log \ell)^{\#\{i=2,3\,:\, q_i=4\}},
		\end{align*}	
		which gives \eqref{asymp2}.		
		Finally, \eqref{asymp2} holds also if $q_3=2$ and $q_1,q_2\geq 3$: applying Lemma \ref{B} we have
		\begin{align*}
		&|\mathfrak{I}_{q_1,\ldots, q_4, \kappa}(\ell)|=\big|\int_{(\SSd)^4} P_\ell(\<x_1, x_2\>)P_\ell(\<x_1, x_4\>)^{q_1-1}P_\ell(\<x_2, x_4\>)^{q_2-1}P_\ell(\<x_3, x_4\>)^2dx\big|\\
		&\leq \frac{C}{\ell^{6}}(\log \ell)^{\#\{i=1,2\,:\, q_i=4\}}  .
		\end{align*}	
		
		The remaining cases (a)--(c) produce an estimate similar to \eqref{asymp2}. As a consequence, \eqref{asymp-R=2} holds in \textbf{Case 1}.

\smallskip

\noindent
\textbf{Case 2:} we deal now with the case when there exists just one connection between the pairs $(1,2)$ and $(3,4)$, that is,
\begin{multicols}{2}
		\begin{enumerate}[label=(\alph*)]
			\item $k_{12}, k_{14}, k_{23}\neq 0$;
			\item $k_{12}, k_{13}, k_{24}\neq 0$;
			\item $k_{34}, k_{13}, k_{24}\neq 0$;
			\item $k_{34}, k_{14}, k_{23}\neq 0$.
		\end{enumerate}  
\end{multicols}
This is graphically shown in Figure 	\ref{R=3-case2}.
\begin{figure}[h]
{\footnotesize 	\begin{center}
		\begin{tikzpicture}
		\node[shape=circle,draw](A) at (1,3) {$1$};
		\node[shape=circle,draw](B) at (3,3) {$2$};
		\node[shape=circle,draw](C) at (3,1) {$3$};
		\node[shape=circle,draw](D) at (1,1) {$4$};
		\draw (A) -- (B);
		\draw (A) -- (D);
		\draw (B) -- (C);
		\draw (2,0) node{(a)};
		\node[shape=circle,draw](A) at (4.5,3) {$1$};
		\node[shape=circle,draw](B) at (6.5,3) {$2$};
		\node[shape=circle,draw](C) at (6.5,1) {$3$};
		\node[shape=circle,draw](D) at (4.5,1) {$4$};
		\draw (A) -- (B);
		\draw (A) -- (C);
		\draw (B) -- (D);
		\draw (5.5,0) node{(b)};
		\node[shape=circle,draw](A) at (8,3) {$1$};
		\node[shape=circle,draw](B) at (10,3) {$2$};
		\node[shape=circle,draw](C) at (10,1) {$3$};
		\node[shape=circle,draw](D) at (8,1) {$4$};
		\draw (C) -- (A);
		\draw (B) -- (D);
		\draw (C) -- (D);
		\draw (9,0) node{(c)};
		\node[shape=circle,draw](A) at (11.5,3) {$1$};
		\node[shape=circle,draw](B) at (13.5,3) {$2$};
		\node[shape=circle,draw](C) at (13.5,1) {$3$};
		\node[shape=circle,draw](D) at (11.5,1) {$4$};
		\draw (D) -- (A);
		\draw (C) -- (B);
		\draw (D) -- (C);
		\draw (12.5,0) node{(d)};
		\end{tikzpicture}
	\end{center}
}
	\caption{\small $R_{q_,\ldots,q_4,\kappa}=3$, Case 2}\label{R=3-case2}
\end{figure}
		We study the instance (d), that is $k_{34}, k_{14}, k_{23}\neq 0$.  As $\kappa= \{k_{ij}\}_{i,j=1}^ 4\in \mathcal{A}_{q_1-1, \ldots, q_4-1}$, it follows that $k_{14}=q_1-1$ and $k_{23}=q_2-1$ and we e notice that $k_{34}$ can be equal to $1$.  We have
		\begin{align*}
		&\mathfrak{I}_{q_1,\ldots, q_4, \kappa}(\ell)=\int_{(\SS2)^4} P_\ell(\<x_1, x_2\>)  P_\ell(\<x_1, x_4\>)^{q_1-1}P_\ell(\<x_2, x_3\>)^{q_2-1}P_\ell(\<x_3, x_4\>)^{k_{34}+1}dx
		\end{align*}
		and we prove that
		\begin{equation}\label{asymp3}
		|\mathfrak{I}_{q_1,\ldots, q_4, \kappa}(\ell)|\leq \frac{C_2}{\ell^{6}} \times \ell^{\#\{i=1,2\,:\, q_i=2\}} \,(\log \ell)^{\#\{i=1,2\,:\, q_i=4\}+\1_{k_{34}=2}}.
		\end{equation}
		For $q_1, q_2\geq 3$, applying Lemma \ref{B}, we have immediately
		\begin{align*}
		|\mathfrak{I}_{q_1,\ldots, q_4, \kappa}(\ell)|\leq \frac{C}{\ell^{6}} \times \ell^{\#\{i=1,2\,:\, q_i=2\}} \, (\log \ell)^{\#\{i=1,2\,:\, q_i=4\}+\1_{k_{34}=2}} .
		\end{align*}		
		
		If $q_1=q_2=2$ we have
		\begin{align*}
		|\mathfrak{I}_{q_1,\ldots, q_4, \kappa}(\ell)|&=\big|\int_{(\SS2)^3} P_\ell(\<x_2, x_3\>) P_\ell(\<x_3, x_4\>)^{k_{34}+1}dx_2 dx_3 dx_4 \int_{\SS2}P_\ell(\<x_1, x_2\>)  P_\ell(\<x_1, x_4\>)dx_1\big|\\
		&=\Big(\frac{C}{2\ell+1}\Big)^2\big|\int_{(\SSd)^2} P_\ell(\<x_3, x_4\>)^{k_{34}+2} dx_3 dx_4 \big|
		\leq \frac{C}{\ell^{4}}\, (\log \ell )^{\1_{k_{34}=2}},
		\end{align*}
		and if $q_1=2$ and $q_2\geq 3$ (or $q_2=2$ and $q_1\geq 3$) we can apply Lemma \ref{I} and we have
		\begin{align*}
		|\mathfrak{I}_{q_1,\ldots, q_4, \kappa}(\ell)|&=\big|\int_{(\SSd)^3} P_\ell(\<x_2, x_3\>)^{q_2-1} P_\ell(\<x_3, x_4\>)^{k_{34}+1}dx_2 dx_3 dx_4 \int_{\SSd}P_\ell(\<x_1, x_2\>)  P_\ell(\<x_1, x_4\>)dx_1\big|\\
		&=\frac{C}{2\ell+1}\big|\int_{(\SS2)^3} P_\ell(\<x_2, x_3\>)^{q_2-1}P_\ell(\<x_2, x_4\>) P_\ell(\<x_3, x_4\>)^{k_{34}+1}dx_2 dx_3 dx_4\big|\\
		&\leq \frac{C}{\ell^{5}}\,(\log \ell)^{\1_{q_2=4}+\1_{k_{34}=2}}.
		\end{align*}
		Then \eqref{asymp3} holds.
		
		By symmetry, the cases (a)--(c) give a similar estimate. By resuming, \eqref{asymp-R=2} is true also in \textbf{Case 2}.

\smallskip
		
\noindent
\textbf{Case 3:} we assume now that either $k_{12}=0$ and $k_{34}=0$, as displayed in Figure \ref{R=3-case3}. 
	\begin{figure}[h]
		{\footnotesize 	\begin{center}
				\begin{tikzpicture}
				\node[shape=circle,draw](A) at (1,3) {$1$};
				\node[shape=circle,draw](B) at (3,3) {$2$};
				\node[shape=circle,draw](C) at (3,1) {$3$};
				\node[shape=circle,draw](D) at (1,1) {$4$};
				\draw (A) -- (D);
				\draw (B) -- (C);
				\draw (B) -- (D);
				\draw (2,0) node{(a)};
				\node[shape=circle,draw](A) at (4.5,3) {$1$};
				\node[shape=circle,draw](B) at (6.5,3) {$2$};
				\node[shape=circle,draw](C) at (6.5,1) {$3$};
				\node[shape=circle,draw](D) at (4.5,1) {$4$};
				\draw (A) -- (D);
				\draw (B) -- (C);
				\draw (A) -- (C);
				\draw (5.5,0) node{(b)};
				\node[shape=circle,draw](A) at (8,3) {$1$};
				\node[shape=circle,draw](B) at (10,3) {$2$};
				\node[shape=circle,draw](C) at (10,1) {$3$};
				\node[shape=circle,draw](D) at (8,1) {$4$};
				\draw (A) -- (C);
				\draw (B) -- (D);
				\draw (A) -- (D);
				\draw (9,0) node{(c)};
				\node[shape=circle,draw](A) at (11.5,3) {$1$};
				\node[shape=circle,draw](B) at (13.5,3) {$2$};
				\node[shape=circle,draw](C) at (13.5,1) {$3$};
				\node[shape=circle,draw](D) at (11.5,1) {$4$};
				\draw (A) -- (C);
				\draw (B) -- (D);
				\draw (B) -- (C);
				\draw (12.5,0) node{(d)};
				\end{tikzpicture}
			\end{center}
		}
		\caption{\small $R_{q_1,\ldots,q_4,\kappa}=3$, Case 3}\label{R=3-case3}
	\end{figure}
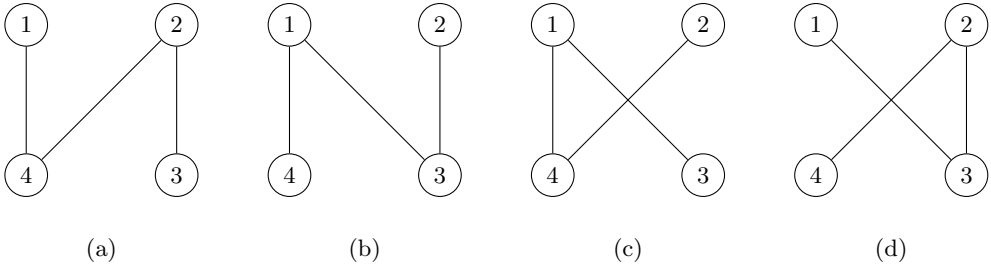

\noindent
This means that
\begin{multicols}{2}
		\begin{enumerate}[label=(\alph*)]
	\item $k_{14}, k_{23}, k_{24}\neq 0$;
	\item $k_{13}, k_{14}, k_{23}\neq 0$;
	\item $k_{13}, k_{14}, k_{24}\neq 0$;
	\item $k_{13}, k_{23}, k_{24}\neq 0$.
\end{enumerate}
\end{multicols}	
		We study the instance (a), that is $k_{14}, k_{23}, k_{24}\neq 0$. Since $\kappa\in \mathcal{A}_{q_1-1,\ldots,q_4-1}$, one gets $k_{14}=q_1-1$ and $k_{23}=q_3-1$. So
		\begin{align*}
		&\mathfrak{I}_{q_1,\ldots, q_4, \kappa}(\ell)=\int_{(\SS2)^4} P_\ell(\<x_1, x_2\>)P_\ell(\<x_1, x_4\>)^{q_1-1}P_\ell(\<x_2, x_3\>)^{q_3-1}P_\ell(\<x_2, x_4\>)^{k_{24}}P_\ell(\<x_3, x_4\>) dx.
		\end{align*}
		We prove that
		\begin{equation}\label{asymp4}
		|\mathfrak{I}_{q_1,\ldots, q_4, \kappa}(\ell)|\leq \frac{C}{\ell^{6}} \times \ell^{\#\{i=1,3\,:\, q_i=2\}} \, (\log \ell)^{\#\{i=1,3\,:\, q_i=4\}+\1_{k_{24}=2}}.
		\end{equation}
		For $q_1, q_3 \geq 3$, applying Lemma \ref{C}, we immediately have
		\begin{align*}
		|\mathfrak{I}_{q_1,\ldots, q_4, \kappa}(\ell)|\leq \frac{C}{\ell^{6}}\,(\log \ell )^{\#\{i=1,3\,:\, q_i=4 \}+\1_{k_{24}=2}}.
		\end{align*}
		If $q_1=q_3 =2$ we have 
		\begin{align*}
		|\mathfrak{I}_{q_1,\ldots, q_4, \kappa}(\ell)|&=\Big|\int_{(\SS2)^2} \int_{\SSd} P_\ell(\<x_1, x_2\>)P_\ell(\<x_1, x_4\>) dx_1 \int_{\SS2} P_\ell(\<x_2, x_3\>) P_\ell(\<x_3, x_4\>) dx_3 \times\\
		&\times P_\ell(\<x_2, x_4\>)^{k_{24}}  dx_2 dx_4\Big|\leq \frac{C}{\ell^{4}} \,(\log \ell )^{\1_{k_{24}=2}}
		\end{align*}
		and \eqref{asymp4} holds.
		If $q_1=2$ and $q_3\geq 3$ (or $q_3=2$ and $q_1\geq 3$), applying Lemma \ref{I}, we have
		\begin{align*}
		&|\mathfrak{I}_{q_1,\ldots, q_4, \kappa}(\ell)|=\Big|\int_{(\SS2)^3} \int_{\SSd} P_\ell(\<x_1, x_2\>)P_\ell(\<x_1, x_4\>) dx_1 P_\ell(\<x_2, x_3\>)^{q_3-1} P_\ell(\<x_3, x_4\>)\times\\
		&\times P_\ell(\<x_2, x_4\>)^{k_{24}}  dx_2 dx_3  dx_4\Big|\\
		&\leq \frac{4\pi}{2\ell+1}\Big|\int_{(\SS2)^3}  P_\ell(\<x_2, x_3\>)^{q_3-1}  P_\ell(\<x_2, x_4\>)^{k_{24}+1} P_\ell(\<x_3, x_4\>) dx_2 dx_3 dx_4\Big|\\
		&\leq \frac{C}{\ell^{5}} \,(\log \ell )^{\1_{q_3=4}+\1_{k_{24}=2}}
		\end{align*}
		and \eqref{asymp4} holds. Again, similar estimates can be produced if (b)--(d) are true. As a consequence, the estimate 
	\eqref{asymp-R=2} holds in \textbf{Case 3} as well.

\smallskip

\noindent
\textbf{Case 4:}  we finally assume that $k_{12},k_{34}\neq 0$, as in Figure \ref{R=3-case4}, that is,
\begin{multicols}{2}
		\begin{enumerate}[label=(\alph*)]
			\item $k_{12}, k_{24}, k_{34}\neq 0$;
			\item $k_{12}, k_{13}, k_{34}\neq 0$;
			\item $k_{12}, k_{23}, k_{34}\neq 0$;
			\item $k_{12}, k_{14}, k_{34}\neq 0$.
		\end{enumerate}
\end{multicols}
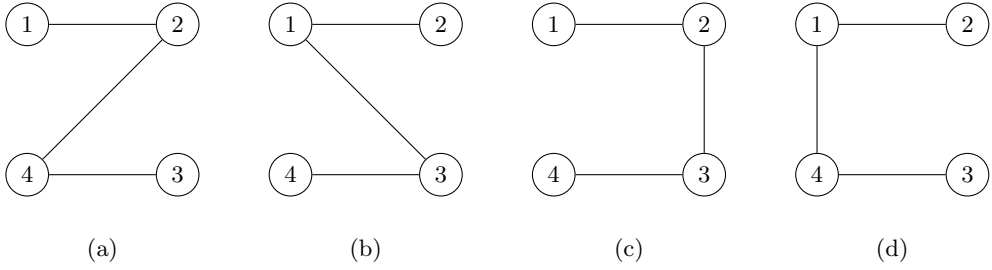
\begin{figure}[h]
{\footnotesize 	\begin{center}
		\begin{tikzpicture}
		\node[shape=circle,draw](A) at (1,3) {$1$};
		\node[shape=circle,draw](B) at (3,3) {$2$};
		\node[shape=circle,draw](C) at (3,1) {$3$};
		\node[shape=circle,draw](D) at (1,1) {$4$};
		\draw (A) -- (B);
		\draw (C) -- (D);
		\draw (B) -- (D);
		\draw (2,0) node{(a)};
		\node[shape=circle,draw](A) at (4.5,3) {$1$};
		\node[shape=circle,draw](B) at (6.5,3) {$2$};
		\node[shape=circle,draw](C) at (6.5,1) {$3$};
		\node[shape=circle,draw](D) at (4.5,1) {$4$};
		\draw (A) -- (B);
		\draw (C) -- (D);
		\draw (A) -- (C);
		\draw (5.5,0) node{(b)};
		\node[shape=circle,draw](A) at (8,3) {$1$};
		\node[shape=circle,draw](B) at (10,3) {$2$};
		\node[shape=circle,draw](C) at (10,1) {$3$};
		\node[shape=circle,draw](D) at (8,1) {$4$};
		\draw (A) -- (B);
		\draw (C) -- (D);
		\draw (B) -- (C);
		\draw (9,0) node{(c)};
		\node[shape=circle,draw](A) at (11.5,3) {$1$};
		\node[shape=circle,draw](B) at (13.5,3) {$2$};
		\node[shape=circle,draw](C) at (13.5,1) {$3$};
		\node[shape=circle,draw](D) at (11.5,1) {$4$};
		\draw (A) -- (B);
		\draw (C) -- (D);
		\draw (A) -- (D);
		\draw (12.5,0) node{(d)};
		\end{tikzpicture}
	\end{center}
}
	\caption{\small  $R=3$, Case 4}\label{R=3-case4}
\end{figure}

We study  $k_{12}, k_{24}, k_{34}\neq 0$. As $\kappa=\{k_{ij}\}_{i,j=1}^ 4\in\mathcal{A}_{q_1-1,\ldots,q_4-1}$, we have $k_{12}=q_1-1$, $k_{34}=q_3-1$. Then
		\begin{align*}
		&\mathfrak{I}_{q_1,\ldots, q_4, \kappa}(\ell)=\int_{(\SS2)^4} P_\ell(\<x_1, x_2\>)^{q_1}P_\ell(\<x_3, x_4\>)^{q_3} P_\ell(\<x_2, x_4\>)^{k_{24}}dx\\
		&= \int_{(\SS2)^2} P_\ell(\<x, y\>)^{q_1}dxdy \int_{(\SS2)^2} P_\ell(\<x, y\>)^{q_3}dxdy \int_{(\SS2)^2} P_\ell(\<x, y\>)^{k_{24}}dxdy. 
		\end{align*}
		If $k_{24}=1$, $I$ is equal to $0$. By using Proposition \ref{GegProp}, we easily obtain 
		\begin{align*}
		|\mathfrak{I}_{q_1,\ldots, q_4, \kappa}(\ell)|\leq \frac{C}{\ell^{6}} \,\1_{k_{24}\neq 1} \times \ell^{\#\{i=1,3\,:\, q_i=2\}+\1_{k_{24}=2}} (\log \ell)^{\#\{ i=1,3\,:\, q_i=4\}+\1_{k_{24}=4}} .
		\end{align*}
		The other instances are similar and the following general estimate holds in \textbf{Case 3}:
		\begin{equation}\label{asymp5-1}
		\mathfrak{I}_{q_1,\ldots, q_4, \kappa}(\ell)\leq 
		\frac{C}{\ell^ {3}},
        \end{equation}
and the proof is concluded.    
	
\end{proof}

\begin{lemma}\label{lemmaI-3}
	
If $R_{q_1,\ldots,q_4,\kappa}=4$, one has
 \begin{equation}\label{asymp-R=4}
|\mathfrak{I}_{q_1,\ldots, q_4, \kappa}(\ell)|\leq \frac{C}{\ell^ {3}},
\end{equation}
	$C>0$ denoting a constant independent of $q_1,\ldots,q_4$ and of $\kappa$.

\end{lemma}

\begin{proof}
We split the proof in 3 different cases according to the properties of $\kappa\in \mathcal{A}_{q_1-1, \ldots, q_4-1}$.

\smallskip

\noindent
\textbf{Case 1:} there exists just one connection either the pair $(1,2)$ or $(3,4)$:
\begin{multicols}{2}
		\begin{enumerate}[label=(\alph*)]
	    \item $k_{14},k_{23}, k_{24}, k_{34}\neq 0$;
		\item $k_{13},k_{23}, k_{24}, k_{34}\neq 0$;
		\item $k_{13},k_{14}, k_{24}, k_{34}\neq 0$;
		\item $k_{13},k_{14}, k_{23}, k_{34}\neq 0$;
		\item $k_{12},k_{14}, k_{23}, k_{24}\neq 0$;
		\item $k_{12},k_{13}, k_{23}, k_{24}\neq 0$;
		\item $k_{12},k_{13}, k_{14}, k_{24}\neq 0$;
		\item $k_{12},k_{13}, k_{14}, k_{23}\neq 0$.
	\end{enumerate}
\end{multicols}
Figure \ref{R=4-case1} shows the cases (a)--(d) where $k_{34}\neq 0$, the latter (e)--(h) turning out by changing the role of the pairs $(3,4)$ and $(1,2)$.  

\begin{center}
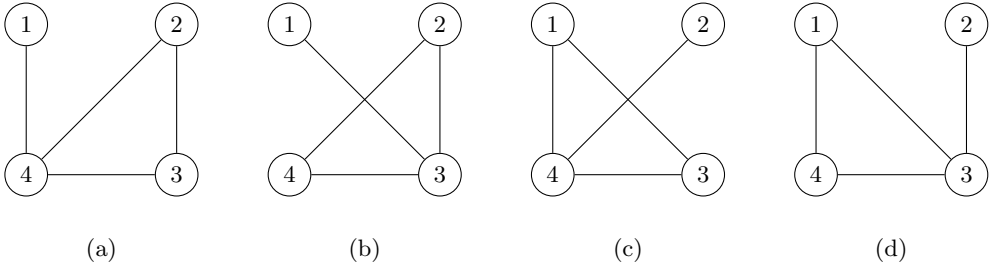
\begin{figure}[h]
	{\footnotesize 	\begin{center}
			\begin{tikzpicture}
			\node[shape=circle,draw](A) at (1,3) {$1$};
			\node[shape=circle,draw](B) at (3,3) {$2$};
			\node[shape=circle,draw](C) at (3,1) {$3$};
			\node[shape=circle,draw](D) at (1,1) {$4$};
			\draw (C) -- (D);
			\draw (A) -- (D);
			\draw (B) -- (D);
			\draw (B) -- (C);
			\draw (2,0) node{(a)};
			\node[shape=circle,draw](A) at (4.5,3) {$1$};
			\node[shape=circle,draw](B) at (6.5,3) {$2$};
			\node[shape=circle,draw](C) at (6.5,1) {$3$};
			\node[shape=circle,draw](D) at (4.5,1) {$4$};
			\draw (C) -- (D);
			\draw (A) -- (C);
			\draw (B) -- (D);
			\draw (B) -- (C);
			\draw (5.5,0) node{(b)};
			\node[shape=circle,draw](A) at (8,3) {$1$};
			\node[shape=circle,draw](B) at (10,3) {$2$};
			\node[shape=circle,draw](C) at (10,1) {$3$};
			\node[shape=circle,draw](D) at (8,1) {$4$};
			\draw (C) -- (D);
			\draw (A) -- (D);
			\draw (B) -- (D);
			\draw (A) -- (C);
			\draw (9,0) node{(c)};
			\node[shape=circle,draw](A) at (11.5,3) {$1$};
			\node[shape=circle,draw](B) at (13.5,3) {$2$};
			\node[shape=circle,draw](C) at (13.5,1) {$3$};
			\node[shape=circle,draw](D) at (11.5,1) {$4$};
			\draw (C) -- (D);
			\draw (A) -- (D);
			\draw (A) -- (C);
			\draw (B) -- (C);
			\draw (12.5,0) node{(d)};
			\end{tikzpicture}
		\end{center}
	}
	\caption{\small  $R_{q_1,\ldots,q_4,\kappa}=4$, Case 1, (a)--(d); for items (e)--(h), just inverting the roles between the pair of nodes $(1,2)$ and $(3,4)$. }\label{R=4-case1}
\end{figure}
\end{center}

We suppose that $k_{14},k_{23}, k_{24}, k_{34} \neq 0$, the other instances being similar. We notice that, in according to $\kappa=\{k_{ij}\}_{i,j=1}^ 4\in \mathcal{A}_{q_1-1,\ldots,q_4-1}$, $k_{14}=q_1-1$
and we can have $k_{23},k_{24},k_{34}\geq 1$. 

When $q_1=2$ we have
\begin{align*}
|\mathfrak{I}_{q_1,\ldots, q_4, \kappa}(\ell)|&=\Big|\int_{(\SS2)^4}P_\ell(\<x_1, x_2\>) P_\ell(\<x_1, x_4\>) P_\ell(\<x_2, x_3\>)^{k_{23}}   P_\ell(\<x_2, x_4\>)^{k_{24}}P_\ell(\<x_3, x_4\>)^{k_{34}+1} dx\Big|\\
&=\frac{4\pi}{2\ell+1}\Big|\int_{(\SS2)^4} P_\ell(\<x_2, x_3\>)^{k_{23}}   P_\ell(\<x_2, x_4\>)^{k_{24}+1} P_\ell(\<x_3, x_4\>)^{k_{34}+1} dx_2 dx_3 dx_4\Big|.
\end{align*}
If $k_{23 }=1$, from Lemma \ref{A}, $|\mathfrak{I}_{q_1,\ldots, q_4, \kappa}(\ell)|\leq \frac{C}{\ell^{5}}(\log \ell )^{\1_{k_{24}=2}+\1_{k_{34}=2}} $, whereas for $k_{23}>1$, we have $|\mathfrak{I}_{q_1,\ldots, q_4, \kappa}(\ell)|\leq \frac{C}{\ell^4}(\sqrt{\log \ell})^{\1_{k_{23}=2}+\1_{k_{24}=1}+\1_{k_{34}=1}}$.
When $q_1\geq 3$, by applying the Cauchy Schwarz inequality, we have
\begin{align*}
&|\mathfrak{I}_{q_1,\ldots, q_4, \kappa}(\ell)|\\
&=\Big|\int_{(\SS2)^4}P_\ell(\<x_1, x_2\>) P_\ell(\<x_1, x_4\>)^{q_1-1} P_\ell(\<x_2, x_3\>)^{k_{23}}   P_\ell(\<x_2, x_4\>)^{k_{24}}P_\ell(\<x_3, x_4\>)^{k_{34}+1} dx\Big|\\
&\leq \int_{(\SS2)^4}| P_\ell(\<x_1, x_2\>) P_\ell(\<x_1, x_4\>)^2 P_\ell(\<x_2, x_3\>)   P_\ell(\<x_2, x_4\>) P_\ell(\<x_3, x_4\>)^{2}| dx\\
&\leq \Big(\int_{(\SS2)^4}P_\ell(\<x_1, x_2\>)^2 P_\ell(\<x_1, x_4\>)^4 P_\ell(\<x_2, x_3\>)^2 dx\Big)^{\frac{1}{2}}\Big( \int_{(\SS2)^4} P_\ell(\<x_2, x_4\>)^2 P_\ell(\<x_3, x_4\>)^{4}dx\Big)^{\frac{1}{2}}\\
&\leq \frac{C_2}{\ell^3 \sqrt{\ell}}\log \ell.
\end{align*}

We can resume by stating the following estimate:
\begin{equation}\label{asymp7-0}
|\mathfrak{I}_{q_1,\ldots, q_4, \kappa}(\ell)|\leq \frac{C}{\ell^3 \sqrt{\ell}}\log \ell.
\end{equation}

\noindent	
\textbf{Case 2:} we assume that $k_{12}, k_{34}\neq 0$ as in Figure \ref{R=4-case2}, that is,
\begin{multicols}{2}
	\begin{enumerate}[label=(\alph*)]
        \item $k_{12},k_{23}, k_{24}, k_{34}\neq 0$;	
	    \item $k_{12},k_{13}, k_{14}, k_{34}\neq 0$;
		\item $k_{12},k_{14}, k_{24}, k_{34}\neq 0$;
		\item $k_{12},k_{13}, k_{23}, k_{34}\neq 0$;
		\item $k_{12},k_{13}, k_{24}, k_{34}\neq 0$;
		\item $k_{12},k_{14}, k_{23}, k_{34}\neq 0$.
	\end{enumerate}
\end{multicols}

\begin{center}
	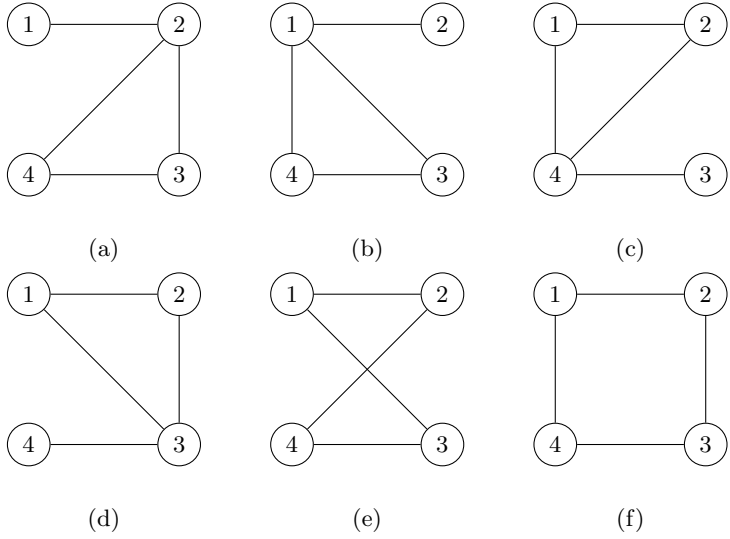
\begin{figure}[h]
		{\footnotesize 	\begin{center}
				\begin{tikzpicture}
				\node[shape=circle,draw](A) at (1,3) {$1$};
				\node[shape=circle,draw](B) at (3,3) {$2$};
				\node[shape=circle,draw](C) at (3,1) {$3$};
				\node[shape=circle,draw](D) at (1,1) {$4$};
				\draw (A) -- (B);
				\draw (C) -- (D);
				\draw (B) -- (D);
				\draw (B) -- (C);
				\draw (2,0) node{(a)};
				\node[shape=circle,draw](A) at (4.5,3) {$1$};
				\node[shape=circle,draw](B) at (6.5,3) {$2$};
				\node[shape=circle,draw](C) at (6.5,1) {$3$};
				\node[shape=circle,draw](D) at (4.5,1) {$4$};
				\draw (A) -- (B);
                \draw (C) -- (D);
				\draw (A) -- (C);
				\draw (A) -- (D);
				\draw (5.5,0) node{(b)};
				\node[shape=circle,draw](A) at (8,3) {$1$};
				\node[shape=circle,draw](B) at (10,3) {$2$};
				\node[shape=circle,draw](C) at (10,1) {$3$};
				\node[shape=circle,draw](D) at (8,1) {$4$};
				\draw (A) -- (B);
				\draw (C) -- (D);
				\draw (A) -- (D);
				\draw (B) -- (D);
				\draw (9,0) node{(c)};
				\end{tikzpicture}

				\begin{tikzpicture}
\node[shape=circle,draw](A) at (1,3) {$1$};
\node[shape=circle,draw](B) at (3,3) {$2$};
\node[shape=circle,draw](C) at (3,1) {$3$};
\node[shape=circle,draw](D) at (1,1) {$4$};
\draw (A) -- (B);
\draw (C) -- (D);
\draw (A) -- (C);
\draw (B) -- (C);
\draw (2,0) node{(d)};
\node[shape=circle,draw](A) at (4.5,3) {$1$};
\node[shape=circle,draw](B) at (6.5,3) {$2$};
\node[shape=circle,draw](C) at (6.5,1) {$3$};
\node[shape=circle,draw](D) at (4.5,1) {$4$};
\draw (A) -- (B);
\draw (C) -- (D);
\draw (A) -- (C);
\draw (B) -- (D);
\draw (5.5,0) node{(e)};
\node[shape=circle,draw](A) at (8,3) {$1$};
\node[shape=circle,draw](B) at (10,3) {$2$};
\node[shape=circle,draw](C) at (10,1) {$3$};
\node[shape=circle,draw](D) at (8,1) {$4$};
\draw (A) -- (B);
\draw (C) -- (D);
\draw (A) -- (D);
\draw (B) -- (C);
\draw (9,0) node{(f)};
\end{tikzpicture}
			\end{center}
		}
		\caption{\small  $R_{q_1,\ldots,q_4,\kappa}=4$, Case 2. }\label{R=4-case2}
	\end{figure}
\end{center}

We study the case $k_{12},k_{23}, k_{24}, k_{34}\neq 0$. 
Then
\begin{align*}
	&\mathfrak{I}_{q_1,\ldots, q_4, \kappa}(\ell)=\int_{(\SSd)^4}P_\ell(\<x_1, x_2\>)^{q_1}  P_\ell(\<x_2, x_3\>)^{k_{23}}   P_\ell(\<x_2, x_4\>)^{k_{24}} P_\ell(\<x_3, x_4\>)^{k_{34}+1}dx\\
	& =\int_{\SSd}P_\ell(\<x_1, x_2\>)^{q_1}dx_1 \int_{(\SSd)^3} P_\ell(\<x_2, x_3\>)^{k_{23}} P_\ell(\<x_2, x_4\>)^{k_{24}} P_\ell(\<x_3, x_4\>)^{k_{34}+1}dx_2dx_3dx_4
	\end{align*}
If $k_{23}=1$, from Lemma \ref{A} we have $|\mathfrak{I}_{q_1,\ldots, q_4, \kappa}(\ell)|\leq \frac{C}{\ell^{4}}(\log \ell)^{\1_{q_1=4}+\1_{k_{24}=3}+\1_{k_{34}=2}}$. Changing the role of $k_{23}$ and $k_{24}$ we obtain the same estimate. 
If $k_{23}, k_{24}>1$, we have $|\mathfrak{I}_{q_1,\ldots, q_4, \kappa}(\ell)|\leq \frac{C}{\ell^4}(\log \ell)^{\1_{q_1=4}+\1_{k_{23}=2}+\1_{k_{24}=2}+\1_{k_{34}=1}}$.  
	
Resuming, we have 
\begin{equation}\label{asymp7-1}
|\mathfrak{I}_{q_1,\ldots, q_4, \kappa}(\ell)|\leq \frac{C}{\ell^4}(\log \ell)^4.
\end{equation}

\noindent
\textbf{Case 3:} we consider when $k_{12}=k_{34}=0$, giving that $k_{13},k_{14}, k_{23}, k_{24}\neq 0$. We have that
\begin{align*}
&|\mathfrak{I}_{q_1,\ldots, q_4, \kappa}(\ell)|\\
&\leq \int_{(\SS2)^4} |P_\ell(\<x_1,x_2\>)P_\ell(\<x_1,x_3\>)^{k_{13}}P_\ell(\<x_1,x_4\>)^{k_{14}}P_\ell(\<x_2,x_3\>)^{k_{23}}P_\ell(\<x_2,x_4\>)^{k_{24}}P_\ell(\<x_3,x_4\>)| dx\\
&\leq  \Big(\int_{(\SS2)^4} P_\ell(\<x_1,x_2\>)^2 P_\ell(\<x_1,x_3\>)^2 P_\ell(\<x_2,x_4\>)^2 dx  \int_{(\SS2)^4} P_\ell(\<x_1,x_4\>)^2 P_\ell(\<x_2,x_3\>)^2 P_\ell(\<x_3,x_4\>)^2 dx\Big)^{\frac{1}{2}}\\
&\leq \frac{C}{\ell^3}.
\end{align*}
Then, by considering also \eqref{asymp7-0} and \eqref{asymp7-1}, the estimate in \eqref{asymp-R=4} holds.

\end{proof}

\begin{lemma}\label{lemmaI-4}
If $R_{q_1,\ldots,q_4,\kappa}=5$, one has
\begin{equation}\label{asymp-R=5}
|\mathfrak{I}_{q_1,\ldots, q_4, \kappa}(\ell)|\leq \frac{C}{\ell^3 \sqrt{\ell}}\log \ell,
\end{equation}
	$C>0$ denoting a constant independent of $q_1,\ldots,q_4$ and of $\kappa$.
	
\end{lemma}
\begin{proof}
First we notice that we can distinguish two different instances: the first one is when $k_{12}, k_{34}\neq 0$, the second one is when $k_{12}=0$ or $k_{34}=0$ .
We start from the first case. Without loss of generality, we can suppose that $k_{13}=0$. Hence, by using the Cauchy Schwarz inequality and Proposition \ref{GegProp},  
\begin{align*}
&|\mathfrak{I}_{q_1,\ldots, q_4, \kappa}(\ell)|\\
&=\Big|\int_{(\SS2)^4} P_\ell(\<x_1, x_2\>)^{k_{12}+1} P_\ell(\<x_1, x_4\>)^{k_{14}}P_\ell(\<x_2, x_3\>)^{k_{23}}P_\ell(\<x_2, x_4\>)^{k_{24}}P_\ell(\<x_3, x_4\>)^{k_{34}+1} \Big|dx\\
&\leq \mbox{\small $\Big(\int_{(\SS2)^4} P_\ell(\<x_1, x_2\>)^{4} P_\ell(\<x_1, x_4\>)^{2}P_\ell(\<x_2, x_3\>)^{2} dx\Big)^{\frac 12}\Big( \int_{(\SS2)^4}P_\ell(\<x_2, x_4\>)^{2}P_\ell(\<x_3, x_4\>)^{4}dx \Big)^{\frac{1}{2}}$}\\
&\leq \frac{C}{\ell^3 \sqrt{\ell}}\log \ell.
\end{align*}

Now we study $k_{12}=0$.  If $q_1=q_2=3$, again by the Cauchy Schwarz inequality and Proposition \ref{GegProp},
\begin{align*}
&|\mathfrak{I}_{q_1,\ldots, q_4, \kappa}(\ell)|\\
&\leq \int_{(\SS2)^4}| P_\ell(\<x_1, x_2\>) P_\ell(\<x_1, x_3\>) P_\ell(\<x_1, x_4\>) P_\ell(\<x_2, x_3\>)  P_\ell(\<x_2, x_4\>)| P_\ell(\<x_3, x_4\>)^2  dx\\
&\leq \mbox{\small $\Big(\int_{(\SS2)^4} P_\ell(\<x_1, x_2\>)^2 P_\ell(\<x_1, x_3\>)^2 P_\ell(\<x_3, x_4\>)^4 dx\Big)^{\frac 12} \Big(\int_{(\SS2)^4} P_\ell(\<x_1, x_4\>)^2 P_\ell(\<x_2, x_3\>)^2  P_\ell(\<x_2, x_4\>)^2 dx\Big)^{\frac{1}{2}}$}\\
&\leq \frac{C}{\ell^3 \sqrt{\ell}}\log \ell.
\end{align*}
The other cases can be handled in the same way and bring to a faster decay, and \eqref{asymp-R=5} follows.

\end{proof}
\begin{lemma}\label{lemmaI-5}
If $R_{q_1,\ldots,q_4,\kappa}=6$, one has
\begin{equation}\label{asymp-R=6}
|\mathfrak{I}_{q_1,\ldots, q_4, \kappa}(\ell)|\leq \frac{C}{\ell^4}\log \ell,
\end{equation}
	$C>0$ denoting a constant independent of $q_1,\ldots,q_4$ and of $\kappa$.
\end{lemma}
\begin{proof}
Suppose that $k_{ij}\neq 0$ for every $i\neq j$. From the Cauchy Schwarz inequality and Proposition \ref{GegProp}, it follows that
\begin{align*}
&|\mathfrak{I}_{q_1,\ldots, q_4, \kappa}(\ell)|\\
&\leq \int_{(\SS2)^4}| P_\ell(\<x_1, x_2\>)^2 P_\ell(\<x_1, x_3\>) P_\ell(\<x_1, x_4\>) P_\ell(\<x_2, x_3\>)  P_\ell(\<x_2, x_4\>)| P_\ell(\<x_3, x_4\>)^2 |dx \\
&\leq \mbox{\small $\Big(\int_{(\SS2)^4}P_\ell(\<x_1, x_2\>)^4 P_\ell(\<x_1, x_3\>)^2  P_\ell(\<x_2, x_4\>)^2 dx\Big)^{\frac{1}{2}}\Big( \int_{(\SS2)^4} P_\ell(\<x_1, x_4\>)^2P_\ell(\<x_2, x_3\>)^2 P_\ell(\<x_3, x_4\>)^4 dx\Big)^{\frac{1}{2}}$ }\\
&\leq \frac{C}{\ell^4}\log \ell.
\end{align*}
The other cases are similar, giving also a faster decay, so \eqref{asymp-R=6} holds.
\end{proof}

\subsection{Proof of Proposition \ref{propDIM2}} \label{B3}
We come back to the estimate \eqref{appoggio} which appears in the proof of Proposition \ref{Conv0}. We rewrite it in dimension 2, that is in terms of the Legendre polynomials, and we use the definition of $\mathfrak{I}_{q_1,\ldots, q_4, \kappa}(\ell)$ in \eqref{I}:
\begin{align*}
&\Var(\sigma_\ell )
\leq \frac{1}{v_{\ell;2}^4} \sum_{q_1,q_2, q_3, q_4\geq 2} \prod_{i=1}^ 4\frac{|b_{q_i}|}{(q_i-1)!} \prod_{r=1}^4 (q_r-1)! \times\nonumber\\
&\times  \sum_{\{k_{i,j}\}^4_{{ i,j=1}} \in \mathcal{C}_{q_1-1,\ldots,q_4-1}  } \prod_{\underset{i<j}{i,j=1}}^4\frac{1}{k_{ij}!} \Big|\int_{(\mathbb S^2)^4 }\prod_{\underset{i<j}{i,j=1}}^4 P_\ell(\<x_i,x_j\>)^{k_{ij}} P_{\ell}(\<x_1,x_2\>)P_\ell(\<x_3,x_4\>)dx\Big|\\
&=
\frac{1}{v_{\ell;2}^4} \sum_{q_1,q_2, q_3, q_4\geq 2} \prod_{i=1}^ 4\frac{|b_{q_i}|}{(q_i-1)!} \prod_{r=1}^4 (q_r-1)! \sum_{\{k_{i,j}\}^4_{{ i,j=1}} \in \mathcal{C}_{q_1-1,\ldots,q_4-1}  } \prod_{\underset{i<j}{i,j=1}}^4\frac{1}{k_{ij}!} |\mathfrak{I}_{q_1,\ldots, q_4, \kappa}(\ell)|.
\end{align*} 
By recalling \eqref{v-ell}, by applying the estimate in Proposition \ref{prop-I} and by using \eqref{PassDel}, we get
\begin{align*}
\Var(\sigma_\ell )
&\leq C\ell^ 2 \times  \frac{1}{\ell^3} 
\sum_{q_1,q_2, q_3, q_4\geq 2} \prod_{i=1}^ 4\frac{|b_{q_i}|}{(q_i-1)!} \E\Big[\prod_{r=1}^4 H_{q_r-1}(Z)\Big]\\
&\leq \frac C\ell \times \E\Big[\Big(\sum_{q \geq 2}\frac{|b_q|}{(q-1)!}H_{q-1}(Z)\Big)^ 4\Big].
\end{align*} 
The last quantity in the above r.h.s. is finite due to Assumption \ref{ASSUMPTION}, and the proof is concluded.

\end{document}